\documentclass[11pt]{amsart}    
\usepackage{geometry}
\usepackage[utf8]{inputenc}
\usepackage[T1]{fontenc}
\usepackage[all,cmtip]{xy}
\newdir{ >}{{}*!/-4pt/@{>}}
\usepackage{color,enumerate}      
\usepackage{comment}
\usepackage{amsthm, amsmath,mathptmx,newtxmath,newtxtext,dsfont}
\usepackage{mathtools}
\usepackage{booktabs}

\usepackage[dvipsnames]{xcolor}   
\usepackage{xr-hyper}
\usepackage[linktocpage=true,colorlinks=true,hyperindex,citecolor=Royal Blue,linkcolor=Royal Blue]{hyperref}   

\usepackage{hyperref}
   
\newtheorem{theorem}{Theorem}[section] 

\newtheorem{lemma}[theorem]{Lemma}
\newtheorem{exercise}[theorem]{Exercise}
 
\theoremstyle{definition}
\newtheorem{definition}[theorem]{Definition}
\newtheorem{example}[theorem]{Example}
\theoremstyle{remark}
\newtheorem{remark}[theorem]{Remark}
\newtheorem{remarks}[theorem]{Remarks}
   
\newtheorem{axiom}{Axiom}

\allowdisplaybreaks[4]

\newcommand{\Ker}{\mathsf{Ker}}
\newcommand{\im}{\mathsf{Im}}
 
\begin{document}

\author{Kishan Kumar Dayaram}
\address[Kishan Kumar Dayaram]{Department of Mathematics and Applied Mathematics,
    University of Johannesburg, Auckland Park Kingsway Campus,
    P.O. Box 524,
    Auckland Park
    2006,
    South Africa.}

\email{220157896@student.uj.ac.za}

\author{Amartya Goswami*}
\address[Amartya Goswami]{Department of Mathematics and Applied Mathematics,
    University of Johannesburg, Auckland Park Kingsway Campus,
    P.O. Box 524,
    Auckland Park
    2006,
    South Africa.}

\address{National Institute for Theoretical and Computational Sciences (NITheCS), South Africa}

\email{agoswami@uj.ac.za}

\author{Zurab Janelidze}

\address[Zurab Janelidze]{Department of Mathematical Sciences, Stellenbosch University, Private Bag X1 Matieland, 7602, South Africa}

\address{National Institute for Theoretical and Computational Sciences (NITheCS), South Africa}

\email{zurab@sun.ac.za}

\author{Diana Ferreira Rodelo}

\address{Department of Mathematics, University of the Algarve, 8005-139 Faro, Portugal and Center for Research and Development in Mathematics and Applications (CIDMA), Department of Mathematics, University of Aveiro, 3810-193 Aveiro, Portugal}

\email{drodelo@ualg.pt} 
\author{Tim\ Van~der Linden}

\address[Tim Van der Linden]{Institut de Recherche en Math\'ematique et Physique, Universit\'e catholique de Louvain, che\-min du cyclotron~2 bte~L7.01.02, B--1348 Louvain-la-Neuve, Belgium}

\email{tim.vanderlinden@uclouvain.be}

\address{Mathematics \& Data Science, Vrije Universiteit Brussel, Pleinlaan 2, B--1050 Brussel, Belgium}

\email{tim.van.der.linden@vub.be}

\thanks{The fourth author acknowledges financial support by CIDMA (https://ror.org/05pm2mw36)
under the Portuguese Foundation for Science and Technology
(FCT, https://ror.org/00snfqn58), Grants UID/04106/2025 (https://doi.org/10.54499/UID/04106/2025)
and UID/PRR/04106/2025 and Centro de Matemática da Universidade de Coimbra (CMUC), funded by the Portuguese Government through FCT/MCTES, DOI 10.54499/UIDB/00324/2020.\\
The fifth author is a Senior Research Associate of the Fonds de la Recherche Scientifique--FNRS. }

\title[Homological lemmas for (non-abelian) group-like structures by diagram chasing\dots]{Homological lemmas for (non-abelian) group-like structures\\ by diagram chasing in a self-dual context}

\begin{abstract}
    Through abelian categories, homological lemmas for modules admit a self-dual treatment, where half of the proof of a lemma is sufficient to prove the full lemma. In this paper we show how the context of a `noetherian form', recently introduced by the second and third authors, allows a self-dual treatment of these lemmas even in the case of non-abelian categories of group-like structures. This context covers a wide range of examples: module categories, the category of groups, of graded abelian groups, the categories of Lie algebras, of cocommutative Hopf algebras, the category of Heyting semilattices, of loops, the dual of the category of pointed sets, the category of modular/distributive lattices and modular connections, the category of sets and partial bijections, and many others. More generally, it includes all semi-abelian and Grandis exact categories.

    \vspace{0.2cm}
    *\;Corresponding author.
\end{abstract}

\makeatletter
\@namedef{subjclassname@2020}{%
    \textup{2020} Mathematics Subject Classification}
\makeatother









\maketitle


\section*{Introduction}
In group theory, certain results may be formulated and proved using subgroups in the place of elements. This observation goes back to the work of E.~Noether and her followers in the 1930's. It was later taken further by S.~Mac Lane, who introduced a method of \emph{chasing subobjects} to prove diagram lemmas of homological algebra in an abelian category, which mimics the method of chasing elements (see \cite{Homology,Cftwm}). 

Discovery and investigation of abelian categories was preceded by S.\ Mac Lane's observation of a number of duality phenomena in group theory (see \cite{duality}). Duality here refers to change of orientation of morphisms in a category. Axioms for an abelian category are invariant under duality, which means that every result established in an abelian category has a `dual' counterpart.

While for any ring $R$, the category of $R$-modules is an abelian category, the same is not true for other group-like structures, which form for instance the category of $R$-Lie algebras, or the category of non-abelian groups themselves. Although S.~Mac~Lane anticipates in \cite{duality} that a self-dual axiomatic framework for encompassing such categories should be possible, a framework with the anticipated functionality was only arrived at much later, in \cite{DNA IV}, building on the development of projective homological algebra by M.~Grandis \cite{Grandis}, and homological \cite{Borceux Bourn} and semi-abelian \cite{semi abelian} categories.

The framework described in \cite{DNA IV} emerges from earlier investigations \cite{DNA 0,DNA I,DNA II,DNA III} (see also \cite{Weighill,chasing subgroups,chimpinde}), which extend the elementary language of a category with abstractly specified subobjects (rather than making use of subobjects defined internally via monomorphisms). These abstract subobjects form a category fibred over the original category. The functor that connects the two categories becomes the key player here. Duality is now in reference to this functor (i.e., it is based on considering the same functor between dual categories). The switch from a category to a functor is what enables to reveal the kind of duality phenomena for non-abelian group-like structures that are necessary for borrowing the subobject chasing technique from the context of an abelian category.

In \cite{DNA IV}, it was shown how one can establish the classical isomorphism theorems in this new framework. It was also suggested in \cite{DNA IV} that the same framework should be suitable for proving homological lemmas. For various concrete lemmas, this has already been checked in \cite{workshop notes, chasing subgroups, ZJCambr}. Many of these proofs are in fact almost a copy of proofs that work in the particular case of Grandis exact categories \cite{Grandis}. However, there is an obstacle in adapting the proof of the Snake Lemma from \cite{Grandis}. This obstacle was first identified by the third and fourth authors, while working on \cite{workshop notes}. It had to do with the construction of the connecting morphism. The missing ingredient, the `pyramid', was developed in \cite{DNA IV} and we recall it in the present paper. This led to a proof of a more general version of the Snake Lemma than the one formulated in \cite{workshop notes} --- see Remark~\ref{snake remark} (note that there are also minor differences in the axioms of a noetherian form and those used in \cite{workshop notes}). The proof was sketched in \cite{ZJCambr} and worked out in full detail in \cite{Day}.

The present paper is a culmination of these developments. There is a large overlap with the unpublished thesis \cite{Day} of the first author (prepared under the supervision of the second and the third authors) and \cite{DNA IV}. Our goal in this paper is to provide researchers working with group-like structures an efficient tool for verifying old or new homological-style diagram lemmas in their specific context. The efficiency comes from the use of duality, as in the abelian context. For example, consider the Short Five Lemma for (non-abelian) groups, which states that in a given commutative diagram
$$\xymatrix{ A\ar[r]^f \ar[d]^s  &B \ar[r]^g \ar[d]^t  &C\ar[d]^u \\
        A' \ar[r]^x & B' \ar[r]^y &C'}$$
where the rows are short exact sequences, we have: $t$ is an isomorphism if both $s$ and $u$ are isomorphisms. In order to prove this, we would only need to prove that $t$ is injective (or surjective) if both $s$ and $u$ are, and the full result would then follow by applying duality, just like in the abelian case!

The language we use in setting up the theory closely follows the language of group theory, thus making it easily adaptable to group-like structures at hand. We demonstrate such adaption on an example of a particular type of group-like structures: Słomiński algebras \cite{slominski}, which are sets equipped with a very general form of (non-abelian) addition and subtraction.

All semi-abelian categories and Grandis exact categories fit the framework of this paper. Among the concrete categories included in these two classes of examples are the category of modules over a ring, the category of groups (as well as the category of Słomiński algebras), the category of graded abelian groups, categories of Lie algebras, the category of cocommutative Hopf algebras over any field, the category of Heyting semilattices, the category of loops, the dual of the category of pointed sets, the category of modular lattices and modular connections, the category of sets and partial bijections, and many others.


\tableofcontents

\section{The Framework}
\label{framework}

In this section, we describe an axiomatic framework for establishing homomorphism theorems, introduced in \cite{DNA IV} (see also the references there). Some of the basic results given in \cite{DNA IV} without proofs are completed here with proofs. This framework is based on a language where primitives are given by
\begin{itemize}
    \item
          groups,
    \item group (homo)morphisms (we will use the terms `morphism' and `homomorphism' interchangeably), as well as codomain/domain of a group morphism,
    \item composition of group morphisms,
    \item subgroup of a group,
    \item subgroup inclusion, and
    \item direct/inverse image of a subgroup along a morphism.
\end{itemize}
When applied to group theory, these terms obtain their usual meaning. The framework can be applied to many other group-like structures as well, where these terms are to be interpreted in a natural way. For instance, to apply the framework to module theory, we must interpret
\begin{itemize}
    \item a group as a module,
    \item (codomain/domain of) a group morphism as (codomain/domain of) a module homomorphism,
    \item composition of morphisms as the usual composition of module homomorphisms,
    \item subgroup of a group as a submodule of a module,
    \item subgroup inclusion as inclusion of submodules,
    \item direct/inverse image of a subgroup along a morphism as direct/inverse image of a submodule along a morphism.
\end{itemize}
The framework can also be applied to more general categories, such as semi-abelian categories in the sense of \cite{semi abelian}. In this case, we interpret
\begin{itemize}
    \item a group as an object in the category,
    \item (codomain/domain of) a group morphism as (codomain/domain of) a morphism between objects,
    \item composition of morphisms as the usual composition of morphisms in the category,
    \item subgroup of a group as a subobject of an object,
    \item subgroup inclusion as inclusion of subobjects.
    \item direct/inverse image of a subgroup along a morphism as direct/inverse image of a subobject along a morphism (which are given by regular image and pullback, respectively).
\end{itemize}
Among examples of semi-abelian categories are categories of many different kinds of non-abelian group-like structures: modules over a ring; all kinds of algebras over a ring, including Lie algebras, associative algebras, non-associative algebras, etc.; groups; rings without identity; Heyting meet semi-lattices; loops; cocommutative Hopf algebras over any field (see \cite{GKV16, GSV19}), and others. Thus, our framework is applicable to all of these structures. It also includes the framework of exact categories in projective homological algebra due to M.~Grandis \cite{Grandis}. Furthermore, according to \cite{FKvN Phd}, any algebraic category (i.e., any variety of universal algebras) is an example of our framework, although in some cases, such as the category of sets itself, the notion of an exact sequence trivializes (and so, homological lemmas there do not give anything interesting).

This general framework is called a \emph{noetherian form} (see \cite{FKvN Paper,DNA IV}). As explained in \cite{DNA IV}, it can be formalized as a pair of categories and a functor between them, satisfying certain properties. Duality in the elementary theory of a functor (see \cite{Cftwm}) translates to duality of the framework as described in Table~\ref{FigA}. This table displays atomic expressions of the language of a noetherian form in the first column and their dual counterparts in the second column. The logical expressions formed using these atomic expressions are to be dualized by dualizing each of the atomic expressions in it and keeping the rest of the logical structure of an expression unchanged. Notice that double dual will always give back the starting expression.

\begin{table}
    \resizebox{\textwidth}{!}{%
        \begin{tabular}{cc}
            \toprule
            \textbf{Statement}                                                                      & \textbf{Dual Statement}                                                                 \\
            \midrule
            $G$ is a group                                                                          & $G$ is a group                                                                          \\
            \midrule[0pt]
            $S$ is a subgroup of $G$ (written as $S\in\mathsf{Sub} G$)                              & $S$ is a subgroup of $G$ (written as $S\in\mathsf{Sub} G$)                              \\
            \midrule[0pt]
            $f$ is a morphism                                                                       & $f$ is a morphism                                                                       \\
            from the group $X$ (domain) to the group $Y$ (codomain)                                 & from the group $Y$ (domain) to the group $X$ (codomain)                                 \\
            (written as $f\colon X\to Y$)                                                           & (written as $f\colon Y\to X$)
            \\
            \midrule[0pt]
            $S$ is contained in $T$ (written as $S\subseteq T$)                                     &
            $T$ is contained in $S$ (written as $T\subseteq S$)                                                                                                                               \\
            (where $S$ and $T$ are subgroups of a group $G$)                                        &
            (where $S$ and $T$ are subgroups of a group $G$)
            \\
            \midrule[0pt]
            $h=gf$                                                                                  & $h=fg$                                                                                  \\
            (where $f\colon X\rightarrow Y$, $g\colon Y\rightarrow Z$ and $h\colon X\rightarrow Z$) & (where $f\colon Y\rightarrow X$, $g\colon Z\rightarrow Y$ and $h\colon Z\rightarrow X$) \\
            \midrule[0pt]
            $T$ is a direct image of $S$ along $f$ (written as $T=fS$)                              & $T$ is an inverse image of $S$ along $f$ (written as $T=f^{-1}S$)
            \\
            (where $S\in\mathsf{Sub}X$ and $f\colon X\to Y$)                                        & (where $S\in\mathsf{Sub}X$ and $f\colon Y\to X$)                                        \\\bottomrule
        \end{tabular}
    }
    \smallskip
    \caption{Atomic expressions and their duals}
    \label{FigA}
\end{table}

As hinted in the table, for a group $G$ in the abstract framework, we write $\mathsf{Sub}G$ to denote the set of subgroups of $G$. We require that inclusion of subgroups turns $\mathsf{Sub}G$ into a poset. This means that we have the following properties:
\begin{itemize}
    \item[(P1)]  $A\subseteq A$ (reflexivity),
    \item[(P2)]  $A\subseteq B$ and $B\subseteq C$ imply $A\subseteq C$ (transitivity),
    \item[(P3)]  $A\subseteq B$ and $B\subseteq A$ imply $A=B$ (antisymmetry).
\end{itemize}

For a set $\mathcal{S}$ of subgroups of a group $G$, the \emph{join} of $\mathcal{S}$ is an element $T$ of $\mathsf{Sub}G$, usually denoted $T=\bigvee \mathcal{S}$, such that:
\begin{enumerate}
    \item  $S\subseteq T$ for all $S\in \mathcal{S}$ and
    \item if $S\subseteq U$ for all $S\in \mathcal{S}$ and for some $U\in \mathsf{Sub}G$, then $T\subseteq U$.
\end{enumerate}
Similarly, the \emph{meet} of $\mathcal{S}$ is a subgroup $M$, usually denoted $M=\bigwedge \mathcal{S}$, such that:
\begin{enumerate}
    \item $M\subseteq S$ for all $S\in \mathcal{S}$ and
    \item if $N\subseteq S$ for all $S\in \mathcal{S}$ and for some element $N\in \mathsf{Sub}G$, then $N\subseteq M$.
\end{enumerate}
Note that the notions of join and meet are dual to each other. We require that each $\mathsf{Sub}G$ is a bounded lattice, i.e.,
\begin{itemize}
    \item[(BL)] For each finite $\mathcal{S}\subseteq \mathsf{Sub}G$, the join $\bigvee \mathcal{S}$ and meet $\bigwedge \mathcal{S}$ exist.
\end{itemize}
This includes the case when $\mathcal{S}=\varnothing$, in which case the join is the smallest subgroup of $G$ (that which is contained in every other subgroup of $G$), which we denote by $1$, and the meet is the largest subgroup of $G$ (that which contains every other subgroup of $G$), which we denote by the same $G$ (although formally, we distinguish them --- otherwise, by duality, we would have to identify $G$ with its smallest subgroup as well). When $\mathcal{S}$ has non-zero but finitely many elements, e.g., $\mathcal{S}=\{A_1,\dots,A_n\}$, we write $A_1\vee\dots \vee A_n$ for $\bigvee\mathcal{S}$ and $A_1\wedge\dots \wedge A_n$ for $\bigwedge\mathcal{S}$.

For any morphism $f\colon X\rightarrow Y$, the \emph{image} of $f$, written as $\mathsf{Im}f$, is the direct image of the largest subgroup of $X$ under the map $f$ and (dually) the \emph{kernel} of $f$, written as $\mathsf{Ker}f$, is the inverse image of the smallest subgroup of $Y$ under the map $f$. We require that for each $S\in \mathsf{Sub}X$, the direct image $fS$ of $S$ under $f$ is an element of $\mathsf{Sub}Y$, while the inverse image of $T\in\mathsf{Sub}Y$ under $f$ is an element of $\mathsf{Sub}X$. Thus in particular, $\mathsf{Im}f=fX\in\mathsf{Sub} Y$ and $\mathsf{Ker}f=f^{-1}1\in\mathsf{Sub} X$. In fact, we require that the direct image and inverse image maps along a given $f\colon X\rightarrow Y$ form a Galois connection between $\mathsf{Sub}X$ and  $\mathsf{Sub}Y$. This means that the following holds:
\begin{itemize}
    \item[(G)] $fA\subseteq C \Leftrightarrow A\subseteq f^{-1}C$, when $A\in\mathsf{Sub}X$, $C\in\mathsf{Sub}Y$ and $f\colon X\to Y$.
\end{itemize}
The well-known consequences of the definition of a Galois connection give us the following properties:
\begin{itemize}
    \item[(G1)] The direct and inverse image maps are monotone (i.e., they preserve subgroup inclusions).
    \item[(G2)] For a morphism $f\colon X\rightarrow Y$ and subgroups $A$ of $X$ and $B$ of $Y$, we have $ff^{-1}fA=fA$ and $f^{-1}ff^{-1}B=f^{-1}B$.
    \item[(G3)] The direct image map preserves joins of subgroups and the inverse image map preserves meets of subgroups.
    \item[(G4)] Under a morphism, the direct image of the smallest subgroup is the smallest subgroup and the inverse image of the largest subgroup. This is the same as (G3) considered for empty joins and meets.
\end{itemize}

A morphism $f\colon X\to Y$ is said to be a \textit{zero morphism} when $X=\mathsf{Ker }f$. This turns out to be equivalent to requiring $\mathsf{Im}f=1$, by direct-inverse image Galois connection, according to which
$$X\subseteq f^{-1}1=\mathsf{Ker }f\quad\Leftrightarrow\quad \mathsf{Im}f=fX\subseteq 1;$$
note that since $X$ is the largest subgroup of $X$ and $1$ is the smallest subgroup of $Y$, by antisymmetry, $X\subseteq \mathsf{Ker}f$ is equivalent to $X=\mathsf{Ker }f$ and $\mathsf{Im}f\subseteq 1$ is equivalent to $\mathsf{Im}f=1$.
The dual of a zero morphism is a zero morphism; in other words, the notion of a zero morphism is \emph{self-dual}.

A subgroup $S$ of a group $G$ is \textit{normal} if it is the kernel of some morphism $f\colon G\rightarrow H$. The dual of this notion is: $S$ is \textit{conormal} if it is the image of some morphism $g\colon H\rightarrow G$.
In usual group theory, every subgroup is of course conormal (as every subgroup is the image of its associated embedding). This is not a requirement in the abstract context of a noetherian form since if it were, by duality every subgroup would have to be normal, which does not hold for (non-abelian) groups.

For a group $G$, an \emph{identity morphism} of $G$ is a morphism $\mathsf{id}_{G}\colon G\rightarrow G$ such that $\mathsf{id}_{G}f=f$ and $h\mathsf{id}_{G}=h$ for all morphisms $f\colon A\rightarrow G$ and $h\colon G\rightarrow B$. There is at most one identity morphism for each $G$ ($\mathsf{id}_{G}=\mathsf{id}_{G}\mathsf{id}'_{G}=\mathsf{id}'_{G}$), and we require that there is at least one, and hence exactly one:
\begin{itemize}
    \item[(I)] Every group $G$ admits an identity morphism $\mathsf{id}_G$.
\end{itemize}
A morphism $f\colon X\rightarrow Y$ is an \textit{isomorphism} if there exists a morphism $g\colon Y\rightarrow X$ such that $fg=\mathsf{id}_{Y}$ and $gf=\mathsf{id}_{X}$. The notion of an isomorphism is also self-dual.

The following further properties that we require ensure that groups and morphisms between them form a category and mapping each group to its subgroup lattice yields a functor from the category of groups to the category of bounded lattices, with Galois connections as morphisms:
\begin{enumerate}
    \item[(A)]  $f(gh)=(fg)h$ whenever $f\colon X\rightarrow Y$, $g\colon Y\rightarrow Z$ and $h\colon Z\rightarrow G$
          (i.e., composition of morphisms is associative),
    \item[(F1)]$\mathsf{id}_G S=\mathsf{id}_G^{-1}S=S$, when $S\in\mathsf{Sub}G$,
    \item[(F2)] $g(fS)=(gf)S$ and $f^{-1}(g^{-1}T)=(fg)^{-1}T$, when $f\colon X\rightarrow Y$, $S\in\mathsf{Sub}X$ and $T\in\mathsf{Sub}Y$.
\end{enumerate}
Note that each of the requirements (P1-P3), (I), (G), (A), (F1), (F2) are self-dual. They constitute Axiom 1 in \cite{DNA IV}, which states that:

\begin{axiom}\label{ax1} Groups and group morphisms form a category under composition of groups. Moreover, for each group $G$, the subgroups of $G$ together with subgroup inclusions form a poset. Furthermore, each morphism $f\colon X\rightarrow Y$ defines a monotone Galois connection
    \begin{center}
        $\xymatrix {\mathsf{Sub X} \ar @/^/[r]^{f(-)}  &\mathsf{Sub Y} \ar@/^/[l]^{f^{-1}(-)} }$
    \end{center}
    given by direct and inverse image maps under $f$ (as left adjoint and right adjoint, respectively), and this defines a functor between the category of groups and the category of posets with Galois connections.
\end{axiom}
A useful consequence of Axiom~\ref{ax1} (more precisely, of (G) and (F1,2)) is the following:
\begin{itemize}
    \item[(GF)] The direct and inverse image maps corresponding to an isomorphism $f$ are the same as the inverse image and direct image map, respectively, corresponding to its inverse $f^{-1}$.
\end{itemize}
The requirement (BL) is part of Axiom 2 in \cite{DNA IV}, which states that:

\begin{axiom}\label{ax2}
    The poset of subgroups of each group is a bounded lattice. Moreover, for a morphism $f\colon X\rightarrow Y$ and subgroups $A\in\mathsf{Sub}X$ and $B\in\mathsf{Sub}Y$, we have $ff^{-1}B=B\wedge \mathsf{Im}f$ and $f^{-1}fA=A\vee \mathsf{Ker }f$.
\end{axiom}


The next lemma gives the `Restricted Modular Law': a restricted version of the standard modular law. Notice that this result follows from Axiom~\ref{ax1} and Axiom~\ref{ax2} only. The proof is reproduced from \cite{DNA IV}.

\begin{lemma}[Restricted Modular Law, \cite{DNA IV}]\label{RML}
    For any three subgroups, $X$, $Y$, and $Z$ of a group $G$, if either $Y$ is normal and $Z$ is conormal, or $Y$is conormal and $X$ is normal, then we have:
    \begin{equation*}
        X\subseteq Z \Rightarrow X \vee(Y\wedge Z)=(X\vee Y)\wedge Z.
    \end{equation*}
\end{lemma}
\begin{proof}
    Note the two cases are dual, so it suffices to prove the first one (where $Y$ is normal and $Z$ is conormal).
    Assume $Y=\Ker g$ and $Z=\im f$ for some morphisms $f$ and $g$ respectively.
    Thus, if $X\subseteq Z=\im f$, we have $X=X\wedge \im f$ and hence:
    \begin{align*}
        X\vee (Y \wedge Z) & = X\vee (\Ker g\wedge \im f)
        =(X\wedge \im f)\vee (\Ker g\wedge \im f)         \\
                           & =ff^{-1}X \vee ff^{-1}\Ker g
        =f(f^{-1}X\vee f^{-1}\Ker g)                      \\
                           & =f(f^{-1}X\vee \Ker gf)
        =f(gf)^{-1}(gf) f^{-1}X                           \\
                           & =f(gf)^{-1}g(ff^{-1}X)
        =f(gf)^{-1}g(X\wedge \im f)                       \\
                           & =f(gf)^{-1}gX
        =ff^{-1}g^{-1}gX                                  \\
                           & =ff^{-1}(X\vee \Ker g)
        =(X\vee \Ker g)\wedge \im f                       \\
                           & =(X\vee Y)\wedge Z.\qedhere
    \end{align*}
\end{proof}

\begin{example}
    The join $X\vee Y$ of subgroups $X$ and $Y$ of an (ordinary) group $G$ is given by the (set-theoretic) intersection of all subgroups which contain $X$ and $Y$. The join may be given by $X\vee Y =XY$ (the product of the subgroups given by $XY=\{ xy\mid x\in X, y\in Y\}$) in the special case when the subgroups commute (i.e., $XY=YX$). The conditions in the statement of the Restricted Modular Law guarantee that either $X$ or $Y$ is normal and hence that $X$ and $Y$ commute. Therefore for ordinary groups, the Restricted Modular Law follows from the modular property of groups which states that if $X$, $Y$ and $Z$ are subgroups of a group $G$ with $X\subseteq Z$, then $(XY) \cap Z= X(Y\cap Z)$ (note that in the case when $Y$ is normal, $Y\cap Z$ also commutes with $X$ since $X\subseteq Z$).
\end{example}

The remaining axioms, along with the Restricted Modular Law, allow us to establish certain isomorphism theorems and diagram lemmas, such as the Butterfly Lemma and the Middle $3\times 3$ Lemma, in this context.

\begin{remark}
    Note that it was missed in \cite{DNA IV} that the Restricted Modular Law is needed for the proof of the Butterfly Lemma: the Restricted Modular Law would be needed to show that the dashed zigzag on the butterfly diagram in the proof does induce a homomorphism.
\end{remark}


In the context of ordinary group theory, for any subgroup $S$ of a group $G$, if we consider $S$ as a group, the embedding morphism $\iota_S \colon S\rightarrow G$ has the following universal property: $\mathsf{Im} \iota_S \subseteq S$ and if $f\colon F\rightarrow G$ is any morphism such that $\mathsf{Im }f\subseteq S$, then there exists a unique morphism $u\colon F\rightarrow S$ such that $f=\iota_{S}u$, i.e., such that the diagram
$$\xymatrix {S \ar[r] ^{\iota_S} & G\\ F\ar@{.>}[u]^u  \ar[ur]_f}$$
commutes. In the context of a noetherian form, we capture this phenomenon as follows:
\begin{definition}
    Given a subgroup $S$ of a group $G$, a morphism $m:M\rightarrow G$ is said to be an \textit{embedding} associated to $S$ if $\im m\subseteq S$ and for any morphism $m'\colon M'\rightarrow G$ such that $\im m'\subseteq S$, there exists a unique homomorphism $u\colon M'\rightarrow M$ such that $m'=mu$.
\end{definition}
The dual notion of this is:
\begin{definition}
    Given a subgroup $S$ of a group $G$, a morphism $e\colon G\rightarrow H$ is said to be a \textit{projection} associated to $S$ if $S\subseteq \Ker e$ and for any morphism $e'\colon G\rightarrow H'$ such that $S\subseteq \Ker e'$, there exists a unique morphism $v\colon H\rightarrow H'$ such that $e'=ve$.
\end{definition}

When $S$ is a conormal subgroup of $G$, we write $\iota_{S}\colon  S/1\rightarrow G$ (sometimes $\iota_{S}\colon  S\rightarrow G$) to denote an embedding associated to $S$, when it exists. Dually, when $S$ is a normal subgroup of $G$, we write $\pi_{S}\colon G\rightarrow G/S$ to denote a projection associated to $S$.

\begin{axiom}\label{ax3} Every conormal subgroup has an embedding associated to it and every normal subgroup has a projection associated to it (in other words, $\iota_S$ and $\pi_S$ are defined whenever $S$ is conormal and normal, respectively).
\end{axiom}

Embeddings allow to define relative normality of subgroups.

\begin{definition}\label{Def Relative Normality}
    For two subgroups $A$ and $B$ of a group $G$, we say that $B$ is \emph{normal to} $A$, denoted $B\lhd A$, if:
    \begin{enumerate}
        \item[(RN1)]  $B\subseteq A$,
        \item[(RN2)]  $A$ is a conormal subgroup of $G$, and,
        \item[(RN3)]  $\iota_{A}^{-1}B$ is normal in $A/1$.
    \end{enumerate}
\end{definition}

To complete the description of the framework, the following axioms remain:

\begin{axiom}\label{ax4}
    Any morphism $f\colon X\rightarrow Y$ factorizes as $f=\iota_{\mathsf{Im }f}h\pi_{\mathsf{Ker }f}$, where $h$ is an isomorphism.
\end{axiom}
\begin{axiom}\label{ax5} The join of any two normal subgroups is normal and the meet of any two conormal subgroups is conormal.
\end{axiom}

The following axiom was not included in \cite{DNA IV} but is necessary to prove homological diagram lemmas. It will be used in the proof on exactly one homological diagram lemma in this paper.
\begin{axiom}\label{ax6}
    Consider any two subgroups $C,N$ of a group $G$ such that $C$ is conormal and $N$ is normal. If $C\subseteq N$, then there exists a binormal subgroup $B$ such that $C\subseteq B\subseteq N$.
\end{axiom}

Notice that almost all examples of noetherian forms considered in \cite{DNA IV} have the property every subgroup is conormal or every subgroup is normal. In such forms, Axiom~\ref{ax6} is automatically satisfied.

\section{An Example: Słomiński Algebras}

The axioms detailed in the previous section have been chosen in such a way that they hold for the context of ordinary groups. Anyone sufficiently familiar with group theory will immediately agree that all stated axioms hold true for groups under the interpretation suggested at the start of the previous section. To give the reader a taste for how little of the group properties is actually required for these axioms to hold, we demonstrate their validity in the wider class of Słomiński algebras \cite{slominski}. We have already mentioned in the previous section various examples that lie within the framework of a noetherian form, along with references where the reader may find out more about these examples. The example of Słomiński algebras falls under the class of examples given by semi-abelian categories --- the category of Słomiński algebras is a semi-abelian category --- so the fact that the axioms presented in the previous section hold for Słomiński algebras is not new.


\begin{definition}
    A \emph{Słomiński algebra} is a set $X$ together with two binary operations, $p$ and $d$, and a constant (nullary operation) $0$ that satisfies the following conditions:
    \begin{enumerate}
        \item $d(x,x)=0$, for all $x\in X$,
        \item $p(d(x,y),y)=x$, for all $x,y\in X$.
    \end{enumerate}
\end{definition}

We now describe an interpretation of a noetherian form.
We shall use some basic notions of universal algebra for Słomiński algebras, such as \emph{subalgebra}, \emph{homomorphism}, \emph{isomorphism} and \emph{congruence}, as well as basic results from universal algebra relating to these notions (see e.g.,~\cite{univalg} for more details).

The subobject poset for a Słomiński algebra $X$ is the poset of its subalgebras. Morphisms in this context will be the homomorphisms between Słomiński algebras. For a homomorphism $f\colon X\rightarrow Y$ between Słomiński algebras, and for subalgebras $A$ of $X$ and $B$ of $Y$, the direct image of $A$ and the inverse image $B$ under $f$ are given as expected, $fA=\{ f(x)\mid x\in A\}$ and $f^{-1}B =\{x\in X \mid  f(x)\in B\}$, respectively. Hence the kernel and image of $f$ are given by $\Ker f=f^{-1}0 =\{ x\in X \mid f(x)=0\}$ and $\im f= f(X) =\{ f(x)\mid x\in X\}$, respectively.

Notice that, for an (ordinary) group $G$ and any two elements $a$ and $b$ of $G$, we define the operation $p$ by $p(a,b)=a\cdot b$ and the operation $d$ by $d(a,b)=a\cdot (b^{-1})$. This makes the group $G$ into a Słomiński algebra. Moreover, Słomiński algebra homomorphisms between groups are precisely group homomorphisms and the subalgebras of a group are subgroups. In this case, kernels and images of homomorphisms obtain their usual meaning.

\begin{lemma}
    Axiom~\ref{ax1} holds for the Słomiński algebra interpretation.
\end{lemma}
\begin{proof}
    The class of Słomiński algebras forms a category since:
    \begin{itemize}
        \item For morphisms $f\colon X\rightarrow Y$ and $g\colon Y\rightarrow Z$ between Słomiński algebras, the composite ${gf\colon X\rightarrow Z}$ is given by $gf(x)=g(f(x))$. This is indeed a morphism since $$gf(d(x,y))=g(f(d(x,y)))=g(d(f(x),f(y)))=d(gf(x), gf(y))$$ and $gf(p(x,y))=p(gf(x), gf(y))$ by a similar calculation.
        \item Composition is associative since $$(h(gf))(x)=h(gf(x))=h(g(f(x)))=hg(f(x))=((hg)f)(x).$$
        \item For a Słomiński algebra $X$, the identity mapping $\mathsf{id}_X \colon X\rightarrow X$ (given by $\mathsf{id}_X (x)=x$) is a morphism since $\mathsf{id}_X (d(x,y))=d(x,y)=d(\mathsf{id}_X (x), \mathsf{id}_X (y))$ and similarly, $\mathsf{id}_X (p(x,y))=p(\mathsf{id}_X (x), \mathsf{id}_X (y))$. Moreover, for any morphism $f\colon X\rightarrow Y$, we have $f(x)=f(\mathsf{id}_X (x))= f\mathsf{id}_X (x)$ and therefore $f=f\mathsf{id}_X$. We also obtain $g=\mathsf{id}_X g$, for any morphism $g\colon Z\rightarrow X$, by a similar argument.
    \end{itemize}
    Moreover, the set of subalgebras, $\mathsf{Sub} X$ of a given algebra $X$, under the operation of set inclusion ($\subseteq$), forms a poset since this is true for any subset of the power set of a set along with set inclusion.

    Notice that both the direct image of a subalgebra $A$ and the inverse image of a subalgebra $B$ are subalgebras, since:
    \begin{itemize}
        \item If  $f(x), f(y)\in fA$, then $$d(f(x),f(y))=f(d(x,y))\in fA$$ and similarly  $$p(f(x),f(y))=f(p(x,y))\in fA .$$
        \item If $x,y\in f^{-1}B$, then $$f(d(x,y))=d(f(x),f(y))\in B$$ and hence $d(x,y)\in f^{-1}B$. Similarly, $p(x,y)\in f^{-1}B$.
    \end{itemize}
    Therefore, the direct and the inverse images of a subalgebra under a morphism $f$ give rise to maps $f(-)\colon \mathsf{Sub}X\rightarrow \mathsf{Sub}Y$ and $f^{-1}(-)\colon \mathsf{Sub Y} \rightarrow \mathsf{Sub} X$,
    $$A\mapsto f(A)=fA,\quad B\mapsto f^{-1}(B)=f^{-1}B.$$
    Moreover, if $fA \subseteq B$, then $f(x)\in B$ for all $x\in A$ and hence $A\subseteq f^{-1}B$. Similarly, if $A\subseteq f^{-1} B$, then $f(x)\in B$ for all $x\in A$ and hence $fA\subseteq B$. Therefore, $fA\subseteq B \Leftrightarrow A\subseteq f^{-1}B$, showing that the maps $f(-)$ and $f^{-1}(-)$ constitute a monotone Galois connection between $\mathsf{Sub}X$ and $\mathsf{Sub}Y$. \qedhere
\end{proof}

\begin{lemma}\label{ax2slo}
    Axiom~\ref{ax2} holds for the Słomiński algebra interpretation.
\end{lemma}
\begin{proof}
    For a Słomiński algebra $X$, the subset $0$ (consisting of just the constant $0$) and the subset $X$ are both subalgebras and therefore the poset $\mathsf{Sub}X$ has bottom and top elements given by $0$ and $X$, respectively.
    Given two subalgebras $C$ and $D$ of $X$, we may define the meet of $C$ and $D$, denoted $C\wedge D$, to be the largest subalgebra of $X$ contained in both $C$ and $D$. It may be verified that this turns out to be the set-theoretic intersection of $C$ and $D$ and thus the meet of two subalgebras exists (since it may be verified that the intersection is closed under the operations $p$ and $d$ and moreover, contains the constant $0$). In a dual manner, we define the join of $C$ and $D$, denoted $C\vee D$, to be the smallest subalgebra of $X$ containing both $C$ and $D$. It may be verified that the join of $C$ and $D$ is given by the intersection of subalgebras of $A$ containing both $C$ and $D$ (note that $X$ is a subalgebra of $X$ that contains $C$ and $D$ and hence this intersection is not empty) and thus the join of two subalgebras exists (since it may be verified that the intersection is closed under the operations $p$ and $d$ and moreover, contains the constant $0$).

    Suppose we are given a morphism $f\colon X\rightarrow Y$ and subalgebras $A$ of $X$ and $B$ of $Y$. We will show that $ff^{-1}B= B\wedge \im f$.
    Let $x\in f^{-1}B$. Then $f(x)\in B$ and hence $ff^{-1}B\subseteq B$. Moreover, $ff^{-1}B\subseteq \im f$ and hence $ff^{-1}B\subseteq B\wedge \im f$. Now take $y\in B\wedge \im f$. Then $y=f(x)$ for some $x\in X$. Therefore $f(x)=y\in B$ and hence $x\in f^{-1}B$. Hence $B\wedge \im f \subseteq ff^{-1}B$ and thus $ff^{-1}B=B\wedge \im f$.

    We will now show that $f^{-1}fA= A\vee \Ker f$.
    Let $x\in A$. Then $f(x)\in fA$ and hence $x\in f^{-1}fA$. Therefore $A\subseteq f^{-1}fA$. Now let $x\in \Ker f$. Then $f(x)=0\in fA$ and hence $x\in f^{-1}fA$. Therefore $\Ker f \subseteq f^{-1}fA$ and hence $A\vee \Ker f\subseteq f^{-1}fA$. Next, let $x\in f^{-1}fA$. Then $f(x)\in fA$ and hence there exists $a\in A$ such that $f(x)=f(a)$. Therefore, $0=d(f(x),f(a))=f(d(x,a))$ and hence $d(x,a)\in \Ker f$. Hence $x=p(d(x,a),a)\in A\vee \Ker f$ and thus $f^{-1}fA\subseteq A\vee \Ker f$. Therefore $f^{-1}fA= A\vee \Ker f$.
\end{proof}

\begin{lemma}\label{xeqy} In every Słomiński algebra,
    $d(x,y)=0\Leftrightarrow x=y $.
\end{lemma}
\begin{proof}
    By the definition of a Słomiński algebra, if $x=y$, then $d(x,y)=0$. It remains to show that if $d(x,y)=0$, then $x=y$. Suppose $d(x,y)=0$, then $x=p(d(x,y),y)=p(0,y)=p(d(y,y),y)=y$.
\end{proof}

\begin{lemma}\label{Słomiński3}
    Axiom~\ref{ax3} holds for the Słomiński algebra interpretation.
\end{lemma}
\begin{proof}
    Consider a subalgebra $A$ of a Słomiński algebra $X$. Define the map $\iota_A\colon A \rightarrow X$ by $\iota_A (x)=x$ for all $x\in A$. Then $\iota_A$ is a morphism since $\iota_A (d(x,y))=d(x,y)=d(\iota_A (x), \iota_A (y))$ and $\iota_A (p(x,y))=p(\iota_A (x) , \iota_A (y))$ by a similar argument. Moreover, given any morphism $f\colon B\rightarrow X$ such that $\im f\subseteq A$, we may define $f'\colon B\rightarrow A$ by $f'(x)=f(x)$. Then clearly $f'$ is a morphism since $f$ is a morphism. Furthermore, $\iota_A f'(x)=\iota_A f(x)=f(x)$ and hence $f=\iota_A f'$. If $f=\iota_A f''$, for a morphism $f''\colon B\rightarrow A/1$, then $\iota_A f' =\iota_A f''$ and hence $f'(x)=f''(x)$ for all $x\in B$. Therefore $f'=f''$ and $f'$ is the unique morphism such that $f=\iota_A f'$. Hence $\iota_A$ is the embedding associated to $A$. Note that every subalgebra is conormal.

    Now suppose $B$ is a normal subalgebra of $X$. Then there must exist a morphism $g\colon X\rightarrow Y$ such that $\Ker g=B$. We may then define a relation on $X$ by $R=\{(x,y)\mid g(x)=g(y)\}$. We will write $y_1Ry_2$ if $(y_1,y_2)\in R$. Then this relation is an equivalence relation since:
    \begin{itemize}
        \item $g(x)=g(x)$ hence $xRx$ for all $x\in X$. Thus the relation is reflexive.
        \item If $xRy$, then $g(x)=g(y)$ and thus $g(y)=g(x)$. Hence $yRx$ and the relation is symmetric.
        \item If $xRy$ and $yRz$, then $g(x)=g(y)$ and $g(y)=g(z)$. Therefore $g(x)=g(z)$ and thus $xRz$. Hence the relation is transitive.
    \end{itemize}
    Now if $g(x_1)=g(y_1)$ and $g(x_2)=g(y_2)$ then $$g(d(x_1,x_2))=d(g(x_1),g(x_2))=d(g(y_1),g(y_2))=g(d(y_1,y_2)) . $$Therefore if $(x_1, y_1),(x_2,y_2)\in R$, then $d((x_1,y_1), (x_2,y_2))=(d(x_1,x_2), d(y_1,y_2))\in R $; if $(x_1, y_1),(x_2,y_2)\in R$, then $p((x_1,y_1), (x_2,y_2))=(p(x_1,x_2),p(y_1,y_2))\in R$, similarly. Therefore $R$ is a congruence. We then consider the set $X/R$ and define the operations $p'([x],[y])=[p(x,y)]$ and $d'([x],[y])=[d(x,y)]$ on $X/R$ and define the constant $0'=[0]$. The operations $p'$ and $d'$ are well-defined since $R$ is a congruence. Then $X/R$ is a Słomiński algebra since:
    \begin{itemize}
        \item $d'([x],[x])=[d(x,x)]=[0]$ and
        \item $p'(d'([x],[y]),[y])=p'([d(x,y)],[y])=[x]$.
    \end{itemize}
    We will then denote $X/R$ as $X/B$ from now on (since the subalgebra $B$ defines the relation). Define a map $\pi_B \colon X\rightarrow X/B$ by $\pi_B (x)=[x]$. It may be verified that $\pi_B$ is a morphism --- this follows from the definitions of $d'$ and $p'$. Now
    \[
        \Ker \pi_B=\{ x\in X\mid[x]=[0]\}=\{x\in X \mid g(x)=g(0)\}=\{x\in X\mid g(x)=0\}=\{x\in X\mid x\in B\}=B.
    \]
    Suppose $h\colon X \rightarrow Y$ is a morphism such that $B\subseteq \Ker h$. Then define a map $h'\colon X/B \rightarrow Y$ by $h'([x])=h(x)$. Note that $h'$ is well-defined since if $[x]=[y]$, then $g(x)=g(y)$ and $0=d(g(x),g(y))=g(d(x,y))$. It follows that $d(x,y)\in B$. Therefore, $d(h'([x]),h'([y]))= d(h(x),h(y))=h(d(x,y))=0$ (since $B\subseteq \Ker h$) and thus $h([x])=h([y])$. Moreover, $h'$ is a morphism since $h'(d([x],[y]))=h'([d(x,y)])=h(d(x,y))=d(h(x),h(y))=d(h'([x]),h'([y]))$ and $h'(p([x],[y]))=p(h'([x]),h'([y]))$ by a similar argument. It is then clear that $h=h'\pi_B$ since $h(x)=h'([x])=h'\pi_B(x)$. Moreover, if $h=h''\pi_B$ then $h''[x]=h'[x]$ and thus $h''=h'$. This shows that $\pi_B$ is a projection associated to $B$.
\end{proof}

\begin{lemma}
    Axiom~\ref{ax4} holds for the Słomiński algebra interpretation.
\end{lemma}
\begin{proof}
    Suppose $f\colon X\rightarrow Y$ is a morphism. We may define a mapping $h\colon X/\Ker f \rightarrow \im f$ by $h([x])=f(x)$. Note that $h$ is well-defined since $[x]=[y]\Rightarrow f(x)=f(y)\Rightarrow h([x])=h([y])$. Moreover, $h$ is a morphism since $$h(d([x],[y]))=h([d(x,y)])=f(d(x,y))=d(f(x),f(y))=d(h([x]), h([y]))$$ and $h(p([x],[y]))=p(h([x]),h([y]))$ (by a similar calculation). Moreover $h$ is surjective since for any $y\in \im f $, there exists $x\in X$ such that $f(x)=y$ and therefore we have $h([x])=y$. Furthermore, if $h([x])=h([y])$, then $f(x)=f(y)$ and thus $[x]=[y]$. This shows that $h$ is injective and therefore $h$ is an isomorphism. Notice that $\iota_{\im f}h \pi_{\Ker f}(x)=\iota_{\im f}h([x])=\iota_{\im f}(f(x))=f(x)$ and hence $f=\iota_{\im f}h\pi_{\Ker f}$. \qedhere
\end{proof}

\begin{lemma}
    Axiom~\ref{ax5} holds for the Słomiński algebra interpretation.
\end{lemma}
\begin{proof}
    Note that we should show that normal subalgebras are stable under projections and conormal subalgebras are stable under embeddings, which is equivalent to Axiom~\ref{ax5}. However, since all subalgebras are conormal, it suffices to show that normal subalgebras are stable under projections.

    Suppose $f\colon X\rightarrow Y$ is a projection (note that $f$ is surjective by Lemma~\ref{Słomiński3}) and $K$ is a normal subalgebra of $X$. Then there exists a morphism $g\colon X\rightarrow Z$ such the $\Ker g=K$. Define a relation $R$ on $Y$ by
    \[
        R=\{(y_1, y_2)\mid \text{$\exists x_1, x_2 \in X$  such that $f(x_1)=y_1$, $f(x_2)=y_2$, $g(x_1)=g(x_2)$} \}.
    \]
    Note that we write $y_1Ry_2$ if $(y_1,y_2)\in R$. Then $R$ is an equivalence relation since:
    \begin{itemize}
        \item for all $y\in Y$, since $f$ is surjective, there exists $x\in X$ such that $f(x)=y$ and clearly $g(x)=g(x)$. This shows that $yRy$ for all $y\in Y$ and hence the relation is reflexive.
        \item If $yRz$, then exist $w$, $x\in X$ such that $f(w)=y$, $f(x)=z$ and $g(w)=g(x)$. Then clearly $zRx$ by definition of $R$ and hence the relation is symmetric.
        \item If $yRz$ and $zRt$, then there exist $v$, $w$, $x\in X$ such that $f(v)=y$, $f(w)=z$, $f(x)=t$ and $g(v)=g(w)$, $g(w)=g(x)$. Then $g(v)=g(x)$ and $yRt$ and hence the relation is transitive.
    \end{itemize}
    Moreover, we will show that $R$ is a congruence. If $(y,z)$, $(t,s) \in R$, then there exist $a$, $b$, $c$, $d\in X$ such that $f(a)=y$, $f(b)=z$, $f(c)=t$, $f(d)=s$ and $g(a)=g(b)$, $g(c)=g(d)$. Then $f(d(a,c))=d(y,t)$, $f(d(b,d))=d(z,s)$ and $g(d(a,c))=g(d(b,d))$. Therefore we obtain $$(d(y,t), d(z,s))\in R.$$ Moreover, $(p(y,t), p(z,s))\in R$ by a similar argument. Therefore $R$ is a congruence. We may then consider the quotient algebra $Y/R$ with constant $0'=[0]$ and binary operations $p'([x],[y])=[p(x,y)]$ and $d'([x],[y])=[d(x,y)]$. The operations $p'$ and $d'$ are well defined since $R$ is a congruence (note that this quotient is a Słomiński algebra by the same argument as in the proof of Lemma~\ref{Słomiński3}). Moreover, the map $\pi\colon Y\rightarrow Y/R$, given by $\pi (y)=[y]$ is a morphism as in the proof of Lemma~\ref{Słomiński3}. Now suppose $y\in \Ker \pi$, then $yR0$. Hence there exists $x$, $u\in X$ such that $f(x)=y$, $f(u)=0$ and $g(x)=g(u)$. Then $y=fx\in f (g^{-1}gu)=f(u\vee \ker g) =f(u) \vee f\Ker g=fK$ by applying Lemma~\ref{ax2slo} and property $(G2)$ of the direct image map. Conversely, if $y\in fK$, then there exists $x\in K$ such that $y=fx$. Furthermore, since $x\in K$, we have $gx=0$. It follows that $yR0$ since $f0=0$ and $g0=0$. Therefore $\Ker \pi=fK$ and hence $fK$ is normal. \qedhere
\end{proof}

\begin{remark}
    Słomiński algebras form only a particular type of a semi-abelian variety of universal algebras. In general, a variety of universal algebras, seen as a category, is semi-abelian if and only if it is a pointed protomodular \cite{Bourn} variety and if and only if its algebraic theory contains a unique constant $0$, binary terms $d_1,\dots, d_n$, and an $(n+1)$-ary term $p$ satisfying the following identities (see \cite{BouJan03}):
    $$
        \left\{
        \begin{array}{l} d_1(x,x)=\dots=d_n(x,x)=0, \\ p(d_1(x,y),\dots,d_n(x,y),y)=x.
        \end{array}\right.
    $$
    Such varieties were first introduced and studied by A.~Ursini \cite{Urs83,Urs94}. As the reader will notice, Słomiński algebras are given by the case when $n=1$ and $p$ and $d=d_1$ are the only basic operations in the variety. All results contained in this section easily generalize to algebras in an arbitrary semi-abelian variety. In fact, semi-abelian varieties are precisely where they generalize: as it follows from \cite{DNA 0}, validity of the last part of Axiom \ref{ax2} would force a variety to be semi-abelian.
\end{remark}

\section{Embedding and Projection Basics}

Now we look at some significant consequences of the axioms from Section \ref{framework}.
Recall firstly the definitions of a monomorphism and an epimorphism from category theory, which are related to the notions of embeddings and projections in this context. A morphism $f\colon X\rightarrow Y$ in a category is a \emph{monomorphism} if $fg_1 =f g_2 \Rightarrow g_1 =g_2$ for any two parallel morphisms $g_1$ , $g_2\colon V\to X$. Dually, a morphism $g\colon X\rightarrow Y$ in a category is an \emph{epimorphism} if $f_1 g = f_2 g \Rightarrow f_1 =f_2$ for any $f_1$, $f_2\colon Y\to Z$.

\begin{lemma}\label{embeddingmono}
    Every embedding is a monomorphism, and dually, every projection is an epimorphism.
\end{lemma}
\begin{proof}
    Suppose $m \colon M\rightarrow G$ is an embedding associated to a subgroup $S$ of a group $G$. We will show that $m$ is a monomorphism. Suppose $m f= m f'$, for parallel morphisms $f$, $f'$ with codomain $S$. It is then clear that $\mathsf{Im}(m f) \subseteq S$, since $\mathsf{Im}m\subseteq S$. Therefore by applying Axiom~\ref{ax3}, we obtain that the morphism $f$ must be the unique morphism such that it makes the diagram
    $$
        \xymatrix {M \ar[r] ^{m} & G\\ F \ar@{.>}[u]^f  \ar[ur]_{mf}}
    $$
    commute. The same is true for $f'$ and so $f=f'$. The statement that a projection is an epimorphism is dual to the statement that an embedding is a monomorphism and hence must also be true.
\end{proof}

\begin{lemma}\label{classificationembeddings}
    Given two embeddings, $m$ and $m'$, associated to the same subgroup $S$ of a group $G$, we have $m'=mi$ for a unique isomorphism $i$. Moreover, if $m\colon M\rightarrow G$ is an embedding associated to $S$, then so is $mi$, for any isomorphism $i$ with codomain $M$.
\end{lemma}
\begin{proof}
    It is clear that $\im m\subseteq S$ and $\im m'\subseteq S$ since $m$ and $m'$ are embeddings. Hence we then obtain $m=m'u$ and $m'=mv$ (by the definition of an embedding) for unique morphisms $u$ and $v$. Therefore $m=m\mathsf{id}_M =m'u=mvu$ and hence $vu=\mathsf{id}_M$ since $m$ is a monomorphism (by Lemma \ref{embeddingmono}). We obtain $uv=\mathsf{id}_{M'}$ by a similar argument. Therefore $u$ and $v$ are isomorphisms. The uniqueness of the isomorphism $u$ in the expression $m'=mu$ follows from the fact that $m'$ is a monomorphism. Now suppose $m\colon M\rightarrow G$ is an embedding associated to $S$ and $i\colon N\rightarrow M$ is an isomorphism. We will show that $mi$ is an embedding associated to $S$. Firstly, since $\im m\subseteq S$, we have $\im (mi)\subseteq S$. Now suppose $f\colon U\rightarrow G$ is a morphism such that $\im f\subseteq S$. Since $m$ is an embedding, we obtain $f=mu$ for a unique morphism $u\colon U\rightarrow M$. Now since $i$ is an isomorphism, there exists a morphism $i'$ such that $ii'=\mathsf{id}_M$ and $i'i=\mathsf{id}_N$. Hence $f=mu=m\mathsf{id}_M u=mii'u=mi(i'u)$. It remains to show that $i'u$ is the unique morphism such that $f=mi(i'u)$. If $f=mij$ for a morphism $j$ with codomain $N$, then $ij=u$ since the morphism $u$ is unique. Hence $j=\mathsf{id}_N j=i'ij=i'u$ and therefore the morphism $i'u$ is unique.
\end{proof}

\begin{lemma}\label{classificationproj}
    Given two projections $e$ and $e'$ associated to the same subgroup $S$, we have $e=ie'$ for a unique isomorphism $i$. Moreover, if $e:G\rightarrow H$ is an embedding associated to a subgroup $S$, then so is the composite $ie$ for any isomorphism $i$ with domain $H$.
\end{lemma}
Note Lemma \ref{classificationproj} is dual to Lemma \ref{classificationembeddings} and hence also holds. Some useful consequences of Lemmas \ref{classificationembeddings} and \ref{classificationproj}  are:
\begin{itemize}
    \item For a conormal subgroup $S$ of a group $G$, the embeddings associated to $S$ are the composites $\iota_{S}i$ where $i$ is any isomorphism with codomain $S/1$.
    \item For a group $G$, $\mathsf{id}_G$ is an embedding associated to $G$ and a projection associated to $1$. Therefore we must have $\iota_G =\mathsf{id}_G i$ for a unique isomorphism $i$. Hence $\iota_G=\mathsf{id}_G i=i$ and therefore $\iota_G$ is an isomorphism. Dually $\pi_1$ is also an isomorphism.
\end{itemize}

\begin{lemma}\label{imembeddding}
    Any embedding is an embedding associated to its image. Moreover, the image of an embedding associated to a conormal subgroup $S$ is the subgroup $S$.
\end{lemma}
\begin{proof}
    Suppose $f$ is an embedding associated to a subgroup $A$. We will show that $f$ is also an embedding associated to $\im f$. Suppose $g$ is a morphism such that $\im g \subseteq \im f$. Then we have $g=g' f$ (for a unique morphism $g'$) since $\im g \subseteq \im f \subseteq A$ and $f$ is an embedding associated to $A$.

    For the last part suppose $S$ is a conormal subgroup of a group $G$. Then there must exist a morphism $f\colon H\rightarrow G$ such that $\im f=S$. By the universal property of $\iota_{S}$, we have $f=\iota_{S}g$ for a unique morphism $g\colon H\rightarrow S/1$. Hence $S=fH=\iota_{S}gH$ and therefore $S\subseteq \im \iota_{S}$. Moreover, by the definition of the embedding $\iota_{S}$, we have $\im \iota_{S}\subseteq S$. Therefore $\im \iota_{S}=S$.
\end{proof}

\begin{lemma}
    Any projection is a projection associated to its kernel. Moreover, the kernel of a projection associated to a normal subgroup $S$ of $G$ is the subgroup $S$.
\end{lemma}
\begin{proof}
    The result is dual to Lemma \ref{imembeddding} and hence must also be true.
\end{proof}

\begin{lemma}\label{kerembedding}\label{improjection}
    A morphism $f$ is an embedding if and only if it has trivial kernel, i.e., $\mathsf{Ker}f=1$. Dually, a morphism is a projection if and only if its image is the largest subgroup of the codomain.
\end{lemma}

\begin{proof}
    Suppose a morphism $f\colon A\rightarrow B$ is an embedding associated to a conormal subgroup $S$. Then we must have $f=\iota_S i$ for an isomorphism $i\colon A\rightarrow S/1$. Accordingly, there must exist a morphism $i'\colon S/1\rightarrow A$ such that $ii'=\mathsf{id}_{S/1}$. Now by Lemma \ref{imembeddding}, $\im \iota_{S}=S$. Moreover, by Axiom~\ref{ax4}, $f$ factorizes as $f =\iota_{\im f} h\pi_{\Ker f}=\iota_{S} h\pi_{\Ker f}$ (since $\im f=S$ by Lemma \ref{imembeddding}) where $h\colon S/\Ker f \rightarrow S $ is an isomorphism. Now by Lemma \ref{embeddingmono}, $f$~is a monomorphism and noting that $f\mathsf{id}_{A} =f=\iota_{S} h\pi_{\Ker f}=\iota_{S}\mathsf{id}_{S/1}h\pi_{\Ker f}=\iota_{S}ii' h\pi_{\Ker f}=fi' h\pi_{\Ker f}$, we have $i' h\pi_{\Ker f}=\mathsf{id}_A$. Hence $\Ker f=\mathsf{id}_A \Ker f =i' h\pi_{\Ker f} \Ker f =i'h 1=1$ and therefore $f$ has trivial kernel.

    Conversely, suppose $\mathsf{Ker }f$ is trivial, while $f$ factorizes as $f=\iota_{\mathsf{Im }f}h\pi_{\mathsf{Ker }f}=\iota_{\im f}h\pi_1$, where $h$ is an isomorphism. Then  if $\mathsf{id}_G=\iota_G \bar{h}\pi_1$, we have that $\bar{h}\pi_1$ is an isomorphism by applying Lemma \ref{classificationembeddings} (since both $\mathsf{id}_G$ and $\iota_G$ are embeddings associated to $G$). Hence $\pi_1$ is an isomorphism. Therefore by Lemma \ref{classificationembeddings}, $f$ is an embedding associated to $\im f$.
\end{proof}

\begin{lemma}\label{classificationiso}
    A morphism is both an embedding and a projection if and only if it is an isomorphism.
\end{lemma}
\begin{proof}
    Given a morphism $f\colon G\rightarrow H$, by Axiom~\ref{ax4}, we may factorize $f$ as $f=\iota_{\im f}h\pi_{\Ker f}$ where $h$ is an isomorphism.
    If $f$ is both an embedding and a projection, then by Lemma \ref{kerembedding} and Lemma \ref{improjection}, $\Ker f=1$ and $\im f=H$. Now by a consequence of Lemma \ref{classificationembeddings} and Lemma \ref{classificationproj}, we obtain that $\iota_H$ and $\pi_1$ are isomorphisms and hence $f=\iota_H h\pi_1$ is an isomorphism. Conversely, suppose $f$ is an isomorphism. Then there must exist a morphism $f'\colon H\rightarrow G$ such that $f'f=\mathsf{id}_G$. Therefore $1=f'1=f'(ff^{-1}1)=\mathsf{id}_Gf^{-1}1=f^{-1}1$ and hence $f$ has trivial kernel. Now by applying duality, we obtain that the image of $f$ is the largest subgroup of $H$. Therefore, by Lemma \ref{kerembedding} and Lemma \ref{improjection}, we obtain that $f$ is both an embedding and a projection.
\end{proof}

\begin{lemma}\label{inj}
    A morphism is an embedding if and only if the corresponding direct image map is injective (equivalently, the inverse image map is surjective). Dually, a morphism is a projection if and only if the corresponding inverse image map is injective (equivalently, the direct image map is surjective).
\end{lemma}
\begin{proof}
    Consider a morphism $f\colon X\to Y$. Suppose the direct image map $f(-)\colon \mathsf{Sub}X\rightarrow \mathsf{Sub}Y$ is injective. Then $f$ must have trivial kernel. It is then an embedding by Lemma~\ref{kerembedding}. Conversely, suppose $f$ is an embedding. By Lemma~\ref{kerembedding}, we know that $f$ has trivial kernel. Suppose $f(U)=f(V)$ for two subgroups $U$ and $V$ of $X$. Then by Axiom~\ref{ax2} we have $U\vee \Ker f =f^{-1} f(U)=f^{-1}f(V)= V\vee \Ker f$. Since $f$ has trivial kernel, it follows that $U=V$ and the direct image map corresponding to $f$ is injective. The rest follows from the fact that in any Galois connection, the left adjoint is injective/surjective if and only if the right adjoint is surjective/injective.
\end{proof}

In the future, when we speak of a morphism $f$ being injective/surjective, we mean that the corresponding direct image map is such. By the lemma above, injective morphisms are the same as embeddings and surjective morphisms are the same as projections. A useful consequence of this (and in fact, already of Lemma~\ref{kerembedding}) is that if a composite $mm'$ is an embedding, then so is $m'$; dually, if a composite $ee'$ is a projection, then so is $e$.

\begin{lemma}\label{unique decomposition}
    Consider a morphism $f\colon X \rightarrow Y$ such that $f=me$, where $m\colon Z\rightarrow Y$ is an embedding and $e\colon X\rightarrow Z $ is a projection. If $f=m'e'$, where $m'\colon Z'\rightarrow Y$ is an embedding and $e'\colon X\rightarrow Z' $ is a projection, then there is an isomorphism $i\colon Z\to Z'$ making the following diagram commute:
    $$\xymatrix{&Z\ar[dd]^i \ar[dr]^m&\\
            X\ar[dr]_{e'} \ar[ur]^e &&Y\\
            &Z'.\ar[ur]_{m'}&}$$
\end{lemma}
\begin{proof}
    Firstly,
    $$\im m =mZ=meX=\im f=m'e'X=m'Z'$$
    and hence there exists a morphism $i\colon Z\rightarrow Z'$ such that $m=m'i$. Moreover, since $m$ and $m'$ are embeddings, we obtain that $i$ is an embedding. Furthermore, we have $m'ie=me=m'e'$ and therefore, since $m'$ is a monomorphism, we obtain $e'=ie$. Lastly, since $e$ and $e'$ are projections, we have that $i$ is a projection. Therefore $i$ is an isomorphism and moreover, this isomorphism makes the diagram commute. \qedhere
\end{proof}

As noted in \cite{DNA IV}, the property established in the following lemma is in fact equivalent to Axiom~\ref{ax5} (under the rest of the axioms). The proof of this lemma (as well as the subsequent lemma) is identical to the one given in \cite{DNA IV}.

\begin{lemma}\label{direct image normal}
    Normal subgroups are stable under direct images along projections and conormal subgroups are stable under inverse images along embeddings.
\end{lemma}
\begin{proof}
    Let $N$ be a normal subgroup of $G$ and let $p\colon G\rightarrow H$ be a projection. Then by Axiom~\ref{ax5}, the subgroup $N\vee \mathsf{Ker }p$ is normal.
    Now consider the projection $p'$ associated to $N\vee \mathsf{Ker }p$, whose kernel must be $N\vee \mathsf{Ker }p$ itself.
    Now clearly $\mathsf{Ker }p \subseteq N\vee \mathsf{Ker }p$ and by the universal property of the projection $p$, we must have $p'=vp$ for some morphism $v$. Then,
    \begin{equation*}
        \begin{split}
            pN & =pN\wedge \mathsf{Im }p
            =pp^{-1}pN
            = p(N\vee \mathsf{Ker }p)                 \\
               & =p(\mathsf{Ker }p')
            =p(p^{-1}v^{-1}1)
            =pp^{-1}(\mathsf{Ker }v)                  \\
               & =\mathsf{Ker }v \wedge \mathsf{Im }p
            = \mathsf{Ker }v\wedge H
            =\mathsf{Ker }v.
        \end{split}
    \end{equation*}
    Note that $\mathsf{Im }p=H$ by Lemma 1.3.
    Hence $pN$ is normal since $\mathsf{Ker }v$ is normal. The rest of the proof follows by duality.\qedhere
\end{proof}

\begin{lemma}\label{inverse im normal}
    Consider a morphism $f\colon A\rightarrow B$ and subgroups $X\lhd Y$ of $B$ (see Definition~\ref{Def Relative Normality}). If $f^{-1}Y$ is conormal, then $f^{-1}X\lhd f^{-1}Y$.
\end{lemma}
\begin{proof}
    Notice that $\im (f\iota_{f^{-1}Y})=ff^{-1}Y=Y\wedge \im f \subseteq Y$. This implies that there exists a morphism $j\colon Y\rightarrow B$ such that $f\iota_{f^{-1}Y}=\iota_Y j$. Then we have $\Ker (\pi_{\iota_Y^{-1} X} j)=j^{-1}\iota_Y^{-1} X =\iota_{f^{-1}Y}^{-1}f^{-1}X$. Therefore $\iota_{f^{-1}Y}^{-1}f^{-1}X$ is a normal subgroup of $f^{-1}Y$ and hence $f^{-1}X\lhd f^{-1}Y$.\qedhere
\end{proof}

\begin{lemma}\label{proj diamond}
    Any two normal subgroups $N$ and $R$ of a group $G$ give rise to a commutative diagram
    $$\xymatrix{&Z&\\
            X\ar[ur]^x&&Y\ar[ul]_y\\
            &G , \ar[uu]^p \ar[ul]^n \ar[ur]_r&}$$
    where $n$, $p$ and $r$ are projections associated to $N$, $N\vee R$ and $R$, respectively. Moreover, in this diagram, $x$ and $y$ are projections and $y^{-1}xS=rn^{-1}S$ for any subgroup $S$ of $X$.
\end{lemma}
\begin{proof}
    Firstly note that $N\subseteq N\vee R=\Ker p$ and hence, by Axiom~\ref{ax3}, there exists a unique morphism $x\colon X\rightarrow Z$ such that $p=xn$. By a similar argument, we may obtain the morphism $y$. We then obtain the desired commutative diagram. We also deduce that $y$ is a projection since $p=yr$, where $p$ and $r$ are both projections. Similarly $x$ is a projection.

    Moreover, by commutativity, we obtain $N\vee R=r^{-1}\Ker y$ and hence $\Ker y=r(N\vee R)=rN\vee rR=rN$ (noting that $\im r=Y$). Then:
    \begin{align*}
        y^{-1}xS=y^{-1}xnn^{-1}S
        =y^{-1}yrn^{-1}S
        =rn^{-1}S\vee \Ker y
        =rn^{-1}S\vee rN
        =r(n^{-1}S\vee N)
        =rn^{-1}S.
    \end{align*}
    Here, the last equality is obtained by noting that $N=\Ker n \subseteq n^{-1}S$.
\end{proof}

\begin{lemma}\label{inverse image embedding}
    Whenever $A\vee B\subseteq S$, where $S$ is a conormal subgroup of a group, then we have:
    $$\iota_S^{-1}(A\vee B)=\iota_S^{-1}A \vee \iota_S^{-1}B.$$
\end{lemma}
\begin{proof}
    Note that $A\subseteq S$ (since $A\subseteq A\vee B\subseteq S$) and $B\subseteq S$ (since $B\subseteq A\vee B\subseteq S$). Hence $A\wedge S=A$ and $B\wedge S=B$. Therefore:
    \begin{align*}
        \iota_S^{-1}A\vee \iota_S^{-1}B & = (\iota_S^{-1}A\vee \iota_S^{-1}B)\vee 1                                                           \\
                                        & =(\iota_S^{-1}A\vee \iota_S^{-1}B)\vee \Ker \iota_{S}        &  & \text{[$\iota_{S}$ is injective]} \\
                                        & =\iota_S^{-1}\iota_S(\iota_S^{-1}A\vee \iota_S^{-1}B)                                               \\
                                        & =\iota_S^{-1}(\iota_s\iota_S^{-1}A\vee \iota_S\iota_S^{-1}B)                                        \\
                                        & =\iota_{S}^{-1}((A\wedge S)\vee (B\wedge S))                                                        \\
                                        & =\iota_S^{-1}(A\vee B).                                      &  & \qedhere
    \end{align*}
\end{proof}

\section{The Pyramid}
In this section, we consider the question when composition of direct and inverse image maps of morphisms does produce a new morphism. This gives us a tool for constructing new morphisms, which proves to be useful for diagram lemmas such as the Snake Lemma.

A diagram $$\xymatrix{X_0 \ar@{-}[r]^{f_1} & X_1 \ar@{-}[r]^{f_2}&X_2\ar@{-}[r]^{f_3} &X_3 \ar@{-}[r]^{f_4}&X_4\ar@{.}[r]&X_{n-1}\ar@{-}[r]^{f_n}&X_n }$$
of groups and morphisms is called a \textit{zigzag}. Note that the arrowheads have been deliberately left out as they are allowed to appear either on the left or on the right.
The groups in a zigzag will be called the \textit{nodes} of the zigzag. The \textit{opposite zigzag} of a given zigzag is simply the zigzag reflected horizontally. A subgroup $T$ of $X_n$ is obtained from a subgroup $S$ of $X_0$ by \textit{chasing} it forward along the zigzag when:
$$T=f_n^{\circ} \cdots f_4^{\circ} f_3^{\circ}f_2^{\circ}f_1^{\circ}S,$$
where $f_i^{\circ}=f_i$ if the arrowhead in the $i$-th position appears on the right and $f_i^{\circ}=f_i^{-1}$ when the arrowhead in the $i$-th position appears on the left. Note that we may obtain a subgroup $V$ of $X_0$ from a subgroup $U$ of $X_n$ by chasing it backwards along the zigzag in a dual manner.
In a zigzag, when all the arrows pointing to the left are isomorphisms, we say the zigzag is \textit{collapsible} and the composite $f_n^{\circ} \cdots f_4^{\circ} f_3^{\circ}f_2^{\circ}f_1^{\circ}$ is said to be the \textit{induced (homo)morphism}, where each $f_i^{\circ}$ is defined as before; however, in this case, $f_i^{-1}$ represents the inverse of the isomorphism $f_i$.

From a given a zigzag, we are able to construct a \textit{pyramid} of morphisms  --- see Figure~\ref{fig:pyramid},
\begin{figure}
    $$\xymatrix@!=5pt{&&&&&&&X_0^n \ar@{-}[dl] \ar@{-}[dr]&&&&&&&\\
        &&&&&&X_0^{n-1}\ar@{.}[ddll] \ar@{-}[dr]&&X_1^{n} \ar@{-}[dl]\ar@{-}[dr]&&&&&&\\
        &&&&&&&X_1^{n-1}\ar@{.}[ddll]\ar@{-}[dr]&&X_2^n \ar@{-}[dl]\ar@{-}[dr]&&&&&\\
        &&&&X_0^4 \ar@{-}[dl]\ar@{-}[dr]&&&&X_2^{n-1}\ar@{.}[ddll]\ar@{-}[dr]&&X_3^n\ar@{-}[dl]\ar@{-}[dr]&&&&\\
        &&&X_0^3 \ar@{-}[dl]\ar@{-}[dr]&&X_1^4 \ar@{-}[dl]\ar@{-}[dr]&&&&X_3^{n-1}\ar@{.}[ddll]\ar@{-}[dr]&&X_4^n\ar@{-}[dl]\ar@{.}[ddrr]&&\\
        &&X_0^2\ar@{-}[dl]\ar@{-}[dr]&&X_1^3\ar@{-}[dl]\ar@{-}[dr]&&X_2^4 \ar@{-}[dl]\ar@{-}[dr]&&&&X_4^{n-1}\ar@{.}[ddll]\ar@{.}[ddrr]&&&&\\
        &X_0^1\ar@{-}[dl]\ar@{-}[dr]&&X_1^2\ar@{-}[dl]\ar@{-}[dr]&&X_2^3\ar@{-}[dl]\ar@{-}[dr]&&X_3^4\ar@{-}[dl]\ar@{-}[dr]&&&&&&X_{n-1}^n\ar@{-}[dr]&\\
        X_0^0 \ar@{-}[rr]&&X_1^1\ar@{-}[rr]&&X_2^2\ar@{-}[rr]&&X_3^3\ar@{-}[rr]&&X_4^4\ar@{.}[rrrr]&&&&X_{n-1}^{n-1}\ar@{-}[ur]\ar@{-}[rr]&&X_n^n}
    $$
    \caption{A Pyramid of Morphisms}
    \label{fig:pyramid}
\end{figure}
which is a commutative diagram where the base of the pyramid is the given zigzag. In the pyramid, each upward pointing arrow is a projection and each downward pointing arrow is an embedding. We reproduce the construction of the pyramid from \cite{DNA IV}.

Given a zigzag $$\xymatrix@=8pt{X_0^0 \ar@{-}[rr]&&X_1^1\ar@{-}[rr]&&X_2^2\ar@{-}[rr]&&X_3^3\ar@{-}[rr]&&X_4^4\ar@{.}[rr]&&X_{n-1}^{n-1}\ar@{-}[rr]&&X_n^n}$$ we construct the triangles in the base of the pyramid
$$\xymatrix{&X^i_{i-1}\ar@{-}[dr]&\\
    X^{i-1}_{i-1}\ar@{-}[rr]^{f_i}\ar@{-}[ur]&&X^i_i}$$
by decomposing the morphism $f_i$ into an embedding $m$, associated to the image of $f_i$, composed with a projection $e$, associated to the kernel of $f_i$. The fact that any morphism can be decomposed in this manner follows from Axiom~\ref{ax4}. Note that since the direction is not indicated in the zigzag, the possibilities for the triangles in the base are
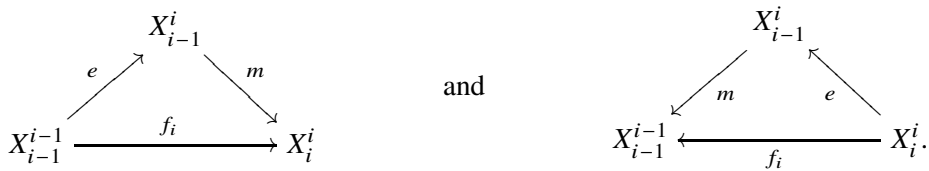
\[
    \vcenter{\xymatrix{&X^i_{i-1}\ar[dr]^m&\\
    X^{i-1}_{i-1}\ar[rr]^{f_i}\ar[ur]^e &&X^i_i}}
    \qquad\qquad\text{and}\qquad\qquad
    \vcenter{\xymatrix{&X^i_{i-1}\ar[dl]^m&\\
    X^{i-1}_{i-1}&&X^i_i\ar[ll]^{f_i} \ar[ul]^e .}}
\]
We then obtain the base
$$
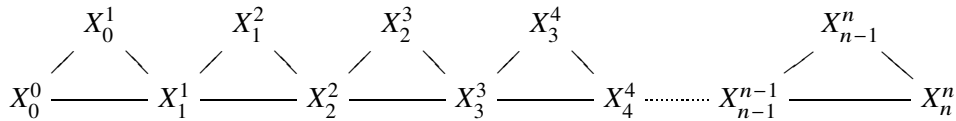
\xymatrix@=8pt{&X_0^1\ar@{-}[dl]\ar@{-}[dr]&&X_1^2\ar@{-}[dl]\ar@{-}[dr]&&X_2^3\ar@{-}[dl]\ar@{-}[dr]&&X_3^4\ar@{-}[dl]\ar@{-}[dr]&& &&X_{n-1}^n\ar@{-}[dl]\ar@{-}[dr]&\\
    X_0^0 \ar@{-}[rr]&&X_1^1\ar@{-}[rr]&&X_2^2\ar@{-}[rr]&&X_3^3\ar@{-}[rr]&&X_4^4\ar@{.}[rr]&&X_{n-1}^{n-1}\ar@{-}[rr]&&X_n^n}$$
of the pyramid. So far the upward pointing arrows are projections and the downward pointing arrows are embeddings.
To build the rest of the pyramid from this base, we consider a wedge
$$
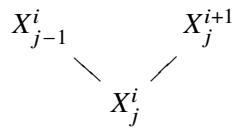
\xymatrix@=8pt{X^i_{j-1}\ar@{-}[dr]&&X^{i+1}_j\ar@{-}[dl]\\
    &X^i_j&}$$
and describe how to build a diamond
$$
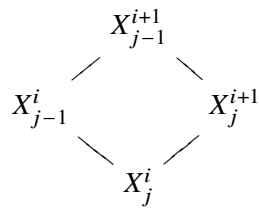
\xymatrix@=8pt{&X^{i+1}_{j-1}\ar@{-}[dl]\ar@{-}[dr]&\\
    X^i_{j-1}\ar@{-}[dr]&&X^{i+1}_j\ar@{-}[dl]\\
    &X^i_j&}$$
(by describing how to construct the top wedge in the diamond). If both arrows in the wedge are pointing upward, they must be projections and we may complete it to a projection diamond by Lemma~\ref{proj diamond}. Dually, if both arrows point downward, they must be embeddings and we may complete it to an embedding diamond by the dual of Lemma~\ref{proj diamond}.

There are still two remaining cases for the given wedge, which are
\[
    \vcenter{\xymatrix@=8pt{X^i_{j-1}\ar[dr]&&X^{i+1}_j\\
    &X^i_j\ar[ur]&}}
    \qquad\qquad\text{and}\qquad\qquad
    \vcenter{\xymatrix@=8pt{X^i_{j-1}&&X^{i+1}_j\ar[dl]\\
    &X^i_j .\ar[ul]&}}
\]
In both cases, we consider the composite of the arrows in the wedge as indicated by the dotted arrows in the wedges
\[
    \vcenter{\xymatrix@=8pt{X^i_{j-1}\ar[dr]\ar@{.>}[rr]&&X^{i+1}_j\\
    &X^i_j\ar[ur]&}}
    \qquad\qquad\text{and}\qquad\qquad
    \vcenter{\xymatrix@=8pt{X^i_{j-1}&&X^{i+1}_j\ar[dl]\ar@{.>}[ll]\\
    &X^i_j\ar[ul]&}}
\]
respectively.
Then, in both cases, we construct the top wedge (and hence the diamond) by decomposing the dotted arrow as an embedding composed with a projection (as we did to build the triangles in the base). This decomposition is again guaranteed by Axiom~\ref{ax4}.

It is then possible to construct the pyramid by repeating the construction of the diamonds to obtain every next layer until we get to the top, which will complete the construction.

Note that in each case of constructing the diamond, we have that the upward pointing arrows are projections and the downward pointing arrows are embeddings. This property is also true for the construction of the triangles in the base of the pyramid. Therefore, all arrows pointing upward in the pyramid are projections and all arrows pointing downward in the pyramid are embeddings.

\begin{figure}
    $$\xymatrix{X_0\ar[d]^{g_0} \ar@{-}[r]^{f_1}&X_1\ar[d]^{g_1}\ar@{-}[r]^{f_2}&X_2\ar[d]^{g_2}\ar@{-}[r]^{f_3}&X_3\ar[d]^{g_3}\ar@{-}[r]^{f_4}&X_4\ar[d]^{g_4}\ar@{.}[r]&X_n \ar[d]^{g_n}\\
            X'_0\ar@{-}[r]^{f'_1}&X'_1\ar@{-}[r]^{f'_2}&X'_2 \ar@{-}[r]^{f_3} &X'_3\ar@{-}[r]^{f_4}&X'_4\ar@{.}[r]&X'_n}$$
    \caption{Diagram connecting two zigzags}
    \label{fig: commutative zigzags}
\end{figure}
Consider a diagram in Figure~\ref{fig: commutative zigzags}, where both the top and the bottom rows are zigzags.
A square $$\xymatrix{X_{i-1}\ar@{-}[r]^{f_i}\ar[d]_{g_{i-1}}&X_i\ar[d]^{g_i}\\
    X'_{i-1}\ar@{-}[r]^{f'_i}&X'_i}$$ from the diagram in Figure~\ref{fig: commutative zigzags} is said to be commutative when
\[
    \begin{cases}
        g_i f_i=f'_i g_{i-1},  & \text{if $f_i$ and $f'_i$ are both right-pointing,}            \\
        g_{i-1}f_i=f'_i g_i,   & \text{if $f_i$ and $f'_i$ are both left-pointing,}             \\
        g_{i-1}=f'_i g_i f_i,  & \text{if $f_i$ is right-pointing and $f'_i$ is left-pointing,} \\
        g_i= f'_i g_{i-1} f_i, & \text{if $f_i$ is left-pointing and $f'_i$ is right-pointing.}
    \end{cases}
\]

\begin{lemma} \label{zigzag iso}
    When in a commutative diagram $$\xymatrix{X_{i-1}\ar@{-}[r]^{f_i}\ar[d]_{g_{i-1}}&X_i\ar[d]^{g_i}\\
        X'_{i-1}\ar@{-}[r]^{f'_i}&X'_i}$$ from the diagram in Figure~\ref{fig: commutative zigzags}, either the top or the bottom arrows are isomorphisms, replacing any of these arrows with their inverses (pointing in the opposite direction), still results in a commutative diagram.
\end{lemma}
\begin{proof}
    Let us assume $f_i$ is an isomorphism. If both $f_i$ and $f'_i$ are right-pointing, then $g_i f_i=f'_i g_{i-1}$. Hence, if we apply the inverse image map of $f_i$ on the right of both sides, we obtain $g_i=f'_i g_{i-1} f_i^{-1}$. This shows that if we replace $f_i$ with its inverse image map then the resulting diagram is commutative. If $f_i$ is left-pointing and $f'_i$ is right-pointing, then $g_i=f'_i g_{i-1} f_i$. If we then apply the inverse image map on the right of both sides, we obtain $g_if_i^{-1}=f'_i g_{i-1}$. The remaining two cases follow by a similar argument to the ones above. Moreover, the case that $f'_i$ is an isomorphism follows by duality. \qedhere
\end{proof}

\begin{lemma}\label{induced morphism}
    In a diagram given by Figure~\ref{fig: commutative zigzags}, where all squares are commutative, if the top and bottom rows are collapsible, then $g_n f=f' g_0$ where $f$ and $f'$ are the morphisms induced by the top and the bottom rows, respectively. Moreover, if all morphisms $g_i$ are isomorphisms such that replacing them with their inverses pointing up retain commutativity in each square in  Figure~\ref{fig: commutative zigzags}, then the top row is collapsible if and only if the bottom row is collapsible.
\end{lemma}
\begin{proof}
    If we apply Lemma~\ref{zigzag iso} to the top and bottom rows, we may then apply commutativity of each square to deduce $g_nf=f' g_0$. We next assume all the morphisms $g_i$ are isomorphisms such that replacing them with their inverses pointing up retain commutativity in each square in  Figure~\ref{fig: commutative zigzags}. Suppose the bottom row is collapsible and consider a left-pointing morphism $f_i$ in the top row. If $f'_i$ is left-pointing (so that, in particular, it is an isomorphism), then it follows that $f_i$ is an isomorphism (since the remaining three morphisms in the commutative square are isomorphisms). If $f'_i$ is right-pointing, then $g_i= f'_i g_{i-1}f_i$ and hence $f_i$ is injective. The commutative square formed by replacing $g_i$ and $g_{i-1}$ with their inverses yields $g_{i-1}^{-1}=f_i g_i^{-1} f'_i$. It follows that $f_i$ is surjective and therefore is an isomorphism. Hence the top row is collapsible.  A similar argument shows that if top row is collapsible, then the bottom row is collapsible.
\end{proof}

\begin{lemma}\label{commutative unique inducedd}
    Given a diagram in Figure~\ref{fig: commutative zigzags}, where all squares commute and the top and the bottom rows are collapsible zigzags, if $g_0$ and $g_n$ are identity morphisms then both collapsible zigzags induce the same morphism.
\end{lemma}
\begin{proof}
    The result follows by setting $g_0=\mathsf{id}_{X_0}$ and $g_n=\mathsf{id}_{X_n}$ and applying Lemma~\ref{induced morphism}. \qedhere
\end{proof}

\begin{lemma}\label{unique pyramid}
    Two pyramids built from the same zigzag admit isomorphisms between the corresponding nodes so that the diagram formed by the two pyramids and the isomorphisms is commutative.
\end{lemma}
\begin{proof}
    The isomorphisms of the triangles in the base are obtained by identity morphisms (to connect the zigzags which are the same) and Lemma~\ref{unique decomposition} (for the top node of the triangle). Furthermore, this guarantees commutativity for the triangles in the base of both pyramids together with the isomorphisms. Moreover, two arrows in corresponding triangles in the base of the pyramid point in the same direction due to the construction of these triangles. If we then assume that the isomorphisms between the nodes in the bottom wedge of a diamond and the diagram formed by the two corresponding wedges together with its isomorphisms is commutative, we will show that the top nodes of the diamonds, formed from these wedges, are isomorphic. Suppose one of the diamonds considered is a projection diamond, then due to the isomorphisms connecting the bottom wedges, the remaining diamond must also be a projection diamond. Moreover, using the commutativity of the bottom wedges and their isomorphisms, we deduce that the direct image of the join of the kernels of the bottom wedge is preserved by the direct image of the isomorphism corresponding to the bottom node. We then conclude that the vertical projection (with the top node as the codomain) composed with the isomorphism of the bottom node yields a projection associated to the same subgroup as the one associated to the vertical projection of the other diamond. Therefore, by applying Lemma~\ref{classificationproj}, we obtain the desired isomorphism. Note that this isomorphism makes the diagram commutative. A dual argument handles the case of embedding diamonds. For the remaining two cases, we apply Lemma~\ref{unique decomposition} to obtain the isomorphism of the top node. Moreover, this also yields commutativity. \qedhere
\end{proof}

\begin{definition}
    A \textit{subquotient} is a zigzag in which all left pointing arrows are embeddings and all right pointing arrows are projections.
\end{definition}

\begin{lemma}\label{collapsible subquotient}
    The opposite zigzag of a subquotient is collapsible if and only if chasing the trivial subgroup of the final node backward along the zigzag results in the trivial subgroup of the initial node.
\end{lemma}
\begin{proof}
    The left pointing arrows in the opposite zigzag of a subquotient are projections. Therefore the opposite zigzag of a subquotient is collapsible if and only each of the projections has trivial kernel (i.e., they are associated to the trivial subgroup of the corresponding node). If the opposite zigzag of a subquotient is collapsible, then it is clear that chasing the trivial subgroup backward along the zigzag results in the trivial subgroup of the initial node (since each right pointing arrow would have trivial kernel). Conversely if we assume that chasing the trivial subgroup of the final subgroup backward along the zigzag results in the trivial subgroup of the initial node, then we obtain that chasing the trivial subgroup of the final node backward along the zigzag up to the node to the right of the initial node results in the trivial subgroup. This is true since if the first morphism was left pointing, it must be an embedding and hence it would have trivial kernel. If it were right pointing, then it must be a projection and hence we may take the direct image map and obtain the desired conclusion (noting that the inverse image maps of projections are right inverses of their direct image maps). By repeating this process, we may deduce that chasing the trivial subgroup of the final node results in the trivial subgroup at each node. Therefore each right pointing morphism must be injective and hence all projections are isomorphisms (since right pointing morphisms are exactly the projections). Hence, the left pointing arrows in the opposite zigzag are isomorphisms.
\end{proof}

\begin{definition}
    Consider a zigzag $$\xymatrix{X_{p_1}^{q_1}\ar@{-}[r]&X_{p_2}^{q_2}\ar@{-}[r]&X_{p_3}^{q_3}\ar@{.}[r]&X_{p_n}^{q_n}}$$ obtained from the pyramid. The zigzag is said to be \textit{horizontal} when $\{p_{i+1}-p_i, q_{i+1}-q_i\}=\{0,1\}$ (i.e., when $p_{i+1}=p_i$ and $q_{i+1}-q_i=1$ or when $p_{i+1}-p_i=1$ and $q_{i+1}=q_i$) for each $i\in \{1,2,\dots ,n-1\}$. The zigzag is said to be \textit{vertical} when $\{p_i-p_{i+1}, q_{i+1}-q_i\}=\{0,1\}$ (i.e., when $p_i=p_{i+1}$ and $q_{i+1}-q_i=1$ or when $p_i-p_{i+1}=1$ and $q_{i+1}=q_i$) for each $i\in \{1,2,\dots ,n-1\}$.
\end{definition}

The horizontal zigzag that joins $X_0^0$ with $X_n^n$, that runs up the left side and down the right side of the pyramid, is said to be the \textit{principal horizontal zigzag} of the pyramid. The vertical zigzag going up along the left side of the triangular outline of the pyramid, which joins $X_0^0$ and $X_n^0$, is said to be the \textit{left principal vertical zigzag} of the pyramid. Similarly, the vertical zigzag going up the right side of the triangular outline of the pyramid, which joins $X_n^n$ with $X_0^n$, is said to be the \textit{right principal vertical zigzag} of the pyramid.

\begin{lemma}\label{vertical subquotients}
    Vertical zigzags are subquotients.
\end{lemma}
\begin{proof}
    By definition of a vertical zigzag, the right pointing arrows coincide with the arrows pointing upward and the left pointing arrows coincide with the arrows pointing downward. The result then follows by noting that, by construction, the upward and downward pointing arrows of the pyramid are projections and embeddings respectively.
\end{proof}

\begin{lemma}
    Chasing a subgroup from one node to another downward via a vertical zigzag and then upward via the same zigzag back to the original node gives back the original subgroup.
\end{lemma}
\begin{proof}
    We know that vertical zigzags are subquotients by Lemma~\ref{vertical subquotients}. Hence chasing a subgroup downward via a vertical zigzag applies direct image maps of embeddings and inverse image maps of projections to the subgroup. If we then chase the subgroup upward via the same zigzag, we would further apply inverse image maps of embeddings and direct image maps of projections. However, since the direct image map of an embedding is a right inverse of its corresponding inverse image map and since the inverse image map of a projection is a right inverse of its corresponding direct image map, we are able to cancel out each direct image map with its inverse image map and we are left with the original subgroup. \qedhere
\end{proof}

\begin{lemma} \label{vertical chasing}
    Chasing a trivial/largest subgroup upward along a vertical zigzag results in a trivial/largest subgroup.
\end{lemma}
\begin{proof}
    Vertical zigzags are subquotients by Lemma~\ref{vertical subquotients}. Therefore, since projections are surjective, we have that chasing the largest subgroup upward along a vertical zigzag results in a largest subgroup. Moreover, since embeddings are injective, chasing a trivial subgroup upward results in a trivial subgroup. \qedhere
\end{proof}

\begin{lemma}\label{vertical collapsible}
    Chasing a trivial/largest subgroup downward/upward along a vertical zigzag results in a trivial/largest subgroup if and only if the zigzag is collapsible in the downward/upward direction.
\end{lemma}
\begin{proof}
    Vertical zigzags are subquotients by Lemma~\ref{vertical subquotients}. The result then follows by applying Lemma~\ref{collapsible subquotient} and its dual. \qedhere
\end{proof}

\begin{lemma}
    Chasing a subgroup from one node to another forward or backward along a horizontal zigzag does not depend on the choice of path.
\end{lemma}
\begin{proof}
    By Lemma~\ref{proj diamond} and its dual, we note that chasing a subgroup along the top wedge of a projection or embedding diamond, is equivalent to chasing it along the bottom. Moreover, this property is clear for the remaining two cases of diamonds in the pyramid. From this, we deduce that the pyramid is commutative in the sense that chasing a subgroup along the top wedge of a diamond yields the same result as chasing along the bottom wedge of the same diamond. Now consider any horizontal zigzag between a node $X_{p_i}^{q_i}$ and $X_{p_j}^{q_j}$. By applying the commutativity of each diamond, for which the horizontal zigzag includes the bottom wedge of the diamond, we obtain that chasing a subgroup along this zigzag yields the same result as chasing the subgroup along the horizontal zigzag that joins $X_{p_i}^{q_i}$ with $X_{p_i}^{q_j}$ followed by chasing along the horizontal zigzag joining $X_{p_i}^{q_j}$ with $X_{p_j}^{q_j}$. Therefore chasing a subgroup along a horizontal zigzag between two nodes does not depend on the choice of path.  \qedhere
\end{proof}

\begin{definition}\label{DefZigzagInduction}
    When the principal horizontal zigzag of a pyramid constructed from a given zigzag is collapsible, the given zigzag is said to \textit{induce a (homo)morphism} and the morphism induced by the principal horizontal zigzag is said to be the \textit{induced (homo)morphism} of the given zigzag.
\end{definition}

\begin{lemma}
    \label{unique induced}
    For a given zigzag, the induced morphism is unique when it exists (in particular, it does not depend on the actual pyramid).
\end{lemma}
\begin{proof}
    Consider a zigzag which induces a morphism using a specific pyramid constructed from the zigzag. If we have any other pyramid constructed from this zigzag, we are able to connect the two pyramids with isomorphisms to form a commutative diagram (by Lemma~\ref{unique pyramid}). Notice that in each commutative square formed by the isomorphisms, replacing both the isomorphisms with their inverses preserves commutativity (by the construction of the isomorphisms in each case). Therefore, by Lemma~\ref{induced morphism}, we obtain that the principal horizontal zigzag (of the second pyramid) is collapsible. Moreover, the morphisms connecting the initial and final nodes of the zigzag are identity morphisms. Hence, by applying Lemma~ \ref{commutative unique inducedd}, we obtain that the induced morphisms are equal. \qedhere
\end{proof}

The following lemma shows that for a collapsible zigzag, the notion of an induced morphism introduced at the start of this section agrees with the one introduced in Definition~\ref{DefZigzagInduction}.

\begin{lemma}
    When a zigzag is collapsible, the principal horizontal zigzag of the corresponding pyramid is collapsible, and moreover, the morphism induced by the first collapsible zigzag is the same as the one induced by the principal horizontal zigzag.
\end{lemma}
\begin{proof}
    We will first show that morphisms in the pyramid pointing north-west or south-west are isomorphisms. This is true for the triangles in the base since the south-west and north-west pointing arrows arise from the decomposition of left-pointing arrows into an embedding of its image composed with a projection of its kernel. For projection diamonds, if the north-west pointing arrow in the bottom wedge is an isomorphism, then by commutativity of the diamond (guaranteed by Lemma~\ref{proj diamond}), we obtain that the north-west pointing morphism in the top wedge is an isomorphism. By duality, if the south-west pointing morphism in the bottom wedge of an embedding diamond is an isomorphism then the south-west pointing morphism in the top wedge is an isomorphism. In the case of the bottom wedge being a north-west pointing morphism composed with a south-west pointing morphism being isomorphisms, we obtain that the constructed vertical morphism is an isomorphism and hence the north-west and south-west pointing morphisms in the top wedge are isomorphisms. Note that the last case for the diamond does not contain north-west or south-west pointing morphisms and hence we do not need to consider it. This shows that all south-west and north-west pointing morphisms in the pyramid are isomorphisms. Moreover, this shows that the principal vertical zigzag is collapsible and the zigzag induces a morphism (in the sense of the pyramid). Moreover, by commutativity of the pyramid, the morphism induced by the principal horizontal zigzag is the same as the morphism induced by the original zigzag. \qedhere
\end{proof}


\begin{lemma}\label{zigzag inverse}
    When a zigzag and its opposite zigzag are both collapsible, the morphisms induced by both of these zigzags are isomorphisms and are inverses of each other.
\end{lemma}
\begin{proof}
    Suppose we are given a zigzag
    $$\xymatrix{X_0 \ar@{-}[r]^{f_1} & X_1 \ar@{-}[r]^{f_2}&X_2\ar@{-}[r]^{f_3} &X_3 \ar@{-}[r]^{f_4}&X_4\ar@{.}[r]&X_{n-1}\ar@{-}[r]^{f_n}&X_n }$$ such that both the zigzag and its opposite
    $$\xymatrix{X_n \ar@{-}[r]^{f_{n}} & X_{n-1} \ar@{-}[r]^{f_{n-1}}&X_{n-2}\ar@{-}[r]^{f_{n-2}} &X_{n-3} \ar@{-}[r]^{f_{n-3}}&X_{n-4}\ar@{.}[r]&X_{1}\ar@{-}[r]^{f_1}& X_0}$$
    are collapsible. Note that, in the opposite zigzag, each $f_i$ points in the opposite direction to the $f_i$ in the original zigzag.
    Since the original zigzag is collapsible, the arrows in it that point to the left are isomorphisms. Moreover, the arrows pointing to the right in the original zigzag are isomorphisms too, since they point to the left in the opposite zigzag, which is also collapsible. It then follows that morphisms induced by the given zigzag and its opposite
    are isomorphisms that are inverses of each other.
\end{proof}

\begin{theorem}[Homomorphism Induction Theorem, \cite{DNA IV}] \label{HIT}
    For a zigzag to induce a homomorphism it is necessary and sufficient that chasing the trivial subgroup of the initial node forward along the zigzag results in the trivial subgroup of the final node, and chasing the largest subgroup of the final node backward along the zigzag results in the largest subgroup of the initial node. Moreover, when a zigzag induces a homomorphism, the induced homomorphism is unique and the direct and inverse image maps of the induced homomorphism are given by chasing a subgroup forward and backward, respectively, along the zigzag.
\end{theorem}
\begin{proof}
    If we apply commutativity, we may chase subgroups along the principal horizontal zigzag of the pyramid.
    Chasing the trivial subgroup upward along the principal left vertical zigzag results in a trivial subgroup by Lemma~\ref{vertical chasing}. If we then continue by chasing the trivial subgroup downward along the right principal vertical zigzag and obtain a trivial subgroup, then by Lemma~\ref{vertical collapsible}, we deduce that the right principal vertical zigzag is collapsible in the downward direction. Similarly, chasing a largest subgroup upward along a vertical zigzag results in a largest subgroup by Lemma~\ref{vertical chasing}. If we then continue chasing a largest subgroup downward along the left principal zigzag, by the dual of Lemma~\ref{collapsible subquotient}, we obtain that the left principal zigzag is collapsible in the upward direction. Therefore the principal horizontal zigzag is collapsible. Hence, the zigzag induces a morphism and moreover, the induced morphism is unique by Lemma~\ref{unique induced}. Again by commutativity, we may reduce chasing subgroups forward and backwards along the principal horizontal zigzag to chasing along the original zigzag (noting that, for a given morphism, its inverse image coincides with the the direct image of the inverse morphism). \qedhere
\end{proof}


\begin{lemma}\label{induced duality}
    A zigzag induces a homomorphism if and only if the opposite zigzag induces a homomorphism in the dual category. Moreover, the homomorphisms induced in both categories are the same.
\end{lemma}
\begin{proof}
    It follows from the Homomorphism Induction Theorem that a zigzag induces a morphism if and only if in the opposite category, chasing the largest subgroup of the initial node forward yields the largest subgroup of the final node and chasing the smallest subgroup of the final node backward yields the smallest subgroup of the initial node. This is equivalent to the opposite of the zigzag inducing a homomorphism in the opposite category. When this is the case, if we horizontally reflect a  pyramid, constructed from the original zigzag, and consider it in the opposite category, it becomes a pyramid for the opposite zigzag in the opposite category. It follows that the induced morphisms are the same.
\end{proof}

\begin{theorem}[Universal Isomorphism Theorem, \cite{DNA IV}]\label{UIT}
    If the largest and smallest subgroups are preserved by chasing them from each end node of the zigzag to the opposite end node, then it induces a homomorphism and the induced homomorphism is an isomorphism.
\end{theorem}
\begin{proof}
    If we apply the Homomorphism Induction Theorem, Theorem~\ref{HIT}, we obtain that the zigzag induces a morphism that is both injective and surjective. Therefore the zigzag induces an isomorphism. \qedhere
\end{proof}

\begin{remark}
    If a zigzag satisfies the assumptions of the Universal Isomorphism Theorem, then so does the opposite zigzag. Furthermore, it follows from Lemma~\ref{zigzag inverse}, that the isomorphisms induced by the zigzag and its opposite are inverses of each other.
\end{remark}

The following is an application of the Universal Isomorphism Theorem.

\begin{theorem}[\cite{DNA IV}]\label{quotientiso}
    Consider any morphism $f\colon A\rightarrow B$, and subgroups   $\Ker f\subseteq W\subseteq X$
    of $A$, where $X$ is conormal. Then $W\lhd X$ if and only if $fW\lhd fX$, and when this is the case, there is an isomorphism $X/W\approx fX/fW$.
\end{theorem}
\begin{proof}
    First note that since $X$ is conormal, the embedding $\iota_{X}$ exists and moreover, $\iota_{X} X/1=X$. Therefore $fX=\im(f\iota_X)$. We then obtain the commutative diagram
    $$\xymatrix{X/1 \ar[rr]^{\iota_X} \ar[d]^{\pi_{\Ker(f\iota_X)}}&&A\ar[r]^f&B\\
        (X/1)/\Ker (f\iota_X) \ar[rrr]^h&&&fX/1. \ar[u]^{\iota_{fX}}}$$
    The factorization of the morphism $f\iota_X$ is obtained by Axiom~\ref{ax4}; note that the morphism $h$ is an isomorphism. If $W\lhd X$, then $\iota_{X}^{-1}W$ is a normal subgroup of $X/1$. Furthermore, $h\pi_{\Ker(f\iota_X)} \iota_X^{-1} W$ is a normal subgroup of $fX/1$ by Lemma~\ref{direct image normal}. However
    \begin{align*}
        h\pi_{\Ker(f\iota_X)} \iota_X^{-1}W & = h\pi_{\Ker(f\iota_X)} \iota_X^{-1} (W\vee \Ker f)                      &  & \text{[$\Ker f\subseteq W$]}                   \\
                                            & =h\pi_{\Ker(f\iota_X)} \iota_X^{-1} f^{-1}fW                                                                                 \\
                                            & =h\pi_{\Ker(f\iota_X)} \pi_{\Ker (f\iota_X)}^{-1}h^{-1}\iota_{fX}^{-1}fW                                                     \\
                                            & =h(h^{-1} \iota_{fX}^{-1}fW  \wedge \im \pi_{\Ker (f\iota_X)})                                                               \\
                                            & =hh^{-1}\iota_{fX}^{-1}fW                                                &  & \text{[$\pi_{\Ker (f\iota_X)}$ is surjective]} \\
                                            & =\iota_{fX}^{-1}W.
    \end{align*}
    Therefore $fW\lhd fX$. Conversely if $fW\lhd fX$, then $\iota_{fX}^{-1}fW$ is a normal subgroup of $fX/1$. Moreover we have that $$W=W\vee \Ker f =f^{-1}fW \lhd f^{-1}fX=X\vee \Ker f=X,$$
    where the normality is obtained by Lemma~\ref{inverse im normal}.

    Moreover, whenever these equivalent cases occur, the desired isomorphism between $X/W$ and $fX/fW$ is induced by the zigzag
    \[
        X/W\xleftarrow{\pi_{\iota_{X}^{-1}}W} X/1 \xrightarrow{\iota_X} A\xrightarrow{f}B\xleftarrow{\iota_{fX}} fX/1\xrightarrow{\pi_{\iota_{fX}^{-1}}fW} fX/fW.\qedhere
    \]
\end{proof}

Lastly, we generalize the result in \cite{DNA IV}, which gives a necessary and sufficient condition, based on the relations induced by morphisms, for when a zigzag of homomorphisms of groups (in the ordinary sense) induces a morphism. In the more general result presented below, ordinary groups are replaced with Słomiński algebras.

Consider a morphism $f\colon X\rightarrow Y$. Then $f$ defines a relation $\bar{f}$ over the sets $X$ and $Y$ given by $x\bar{f}y \Leftrightarrow f(x)=y$. Furthermore, the opposite relation $\bar{f}^{\mathsf{op}}$ is given by $y\bar{f}^{\mathsf{op}} x \Leftrightarrow f(x)=y$.

Consider a zigzag
$$\xymatrix{X_0 \ar@{-}[r]^{f_1} & X_1 \ar@{-}[r]^{f_2}&X_2\ar@{-}[r]^{f_3} &X_3 \ar@{-}[r]^{f_4}&X_4\ar@{.}[r]&X_{n-1}\ar@{-}[r]^{f_n}&X_n.}$$
We write $f_i^\bullet=\bar{f_i}$ if the arrowhead in the $i$-th place in the zigzag appears on the right and $f_i^\bullet=\bar{f_i}^{\mathsf{op}}$ if the arrowhead in the $i$-th place in the zigzag appears on the left.
The \emph{induced relation} of the zigzag is the relation given by the composite $f_n^\bullet \cdots f_4^\bullet f_3^\bullet f_2^\bullet f_1^\bullet$ of relations $f_i^\bullet$.

The proof of the following theorem is adapted from the proof of the corresponding result for groups in \cite{DNA IV}. Although this falls outside the scope of the present paper, it should be possible to generalize this result to arbitrary semi-abelian categories.

\begin{theorem} \label{induced relation}
    A zigzag of Słomiński algebra morphisms induces a morphism if and only if the induced relation is a function. Moreover, when this is the case, the function is precisely the induced morphism.
\end{theorem}
\begin{proof}
    We first show that any two horizontal zigzags between two nodes in the pyramid induce the same relation. For horizontal zigzags consisting of arrows staying in the base of the pyramid, the result is clear. It then suffices to show that for each diamond in the pyramid, the relation induced by the top wedge is the same as the relation induced by the bottom wedge. We prove this property for projection diamonds only, since the property for embedding diamonds follows by duality and the remaining two cases are clear. Consider a projection diamond
    
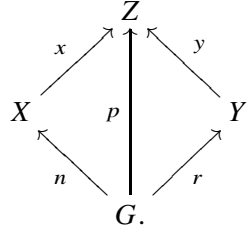
\[
        \xymatrix{&Z&\\
            X\ar[ur]^x&&Y\ar[ul]_y\\
            &G .\ar[uu]^p \ar[ul]^n \ar[ur]_r&}
    \]
    Then it follows, by commutativity, that the relation induced by the bottom wedge is smaller than the one induced by the top wedge. To show that the relation induced by the bottom wedge is bigger than the relation induced by the top wedge, consider elements $a\in X$ and $b\in Y$ such that $x(a)=y(b)$. Then there must exist $c\in G$ such that $n(c)=a$ since $n$ is surjective. Now the element $d(b, r(c))$ must belong to $\Ker y$ since
    \[
        y(d(b,r(c)))=d(y(b), yr(c))=d(y(b), xn(c))=d(y(b),x(a))=0.
    \]
    Moreover,
    \begin{align*}
         & \Ker n\vee \Ker r=\Ker p =\Ker (yr)=r^{-1}\Ker y     \\
         & \Rightarrow  r(\Ker n \vee \Ker r)=rr^{-1}\Ker y     \\
         & \Rightarrow r\Ker n \vee r\Ker r=\Ker y \wedge \im r \\
         & \Rightarrow \Ker y= r\Ker n.
    \end{align*}
    Therefore there exists $d\in \Ker n$ such that $r(d)=d(b,r(c))$. Let $e=p(d,c)$. Then
    \[
        n(e)=n(p(d,c))=p(n(d),n(c))=p(0,a)=a
    \]
    and
    \[
        r(e)=r(p(d,c))=p(r(d),r(c))=p(d(b,r(c)),r(c))=b.
    \]
    Therefore, the relation induced by the bottom wedge is indeed bigger than the relation induced by the top wedge and hence we may conclude that the induced relations are the same. Thus we have shown that any two horizontal zigzags between two nodes in the pyramid induce the same relation. This means that the relation induced by a zigzag is the same as the relation induced by the corresponding principal horizontal zigzag. If the principal horizontal zigzag is collapsible, then it is clear that the induced relation is the induced morphism. Conversely, suppose that the induced relation is a function. We will show that the principal horizontal zigzag is collapsible. Consider the first left-pointing arrow (which is an embedding) in the left principal vertical zigzag. Then any element in its codomain must be related to an element in its domain since the previous morphisms in the left principal vertical zigzag are surjections and since the element must be related to an element of $X_n^n$. Therefore the arrow is surjective and is hence an isomorphism. A similar argument shows that all left pointing arrows in the zigzag connecting $X_n^n$ and $X_0^n$ are isomorphisms.
\end{proof}

\section{Classical Diagram Lemmas} \label{diaglem}
In this section, we state and prove the classical diagram lemmas of homological algebra (found, e.g., in \cite{Homology,Cftwm}) in the self-dual axiomatic context described in Section \ref{framework}. The diagram lemmas will be formulated in terms of exact sequences and injective and surjective morphisms, where in this context injectivity and surjectivity refer to the injectivity and surjectivity of the direct image maps. Notice that injective morphisms are precisely embeddings and dually, surjective morphisms are precisely projections by Lemma~\ref{inj}. Most of these lemmas can be broken up into two dual halves, which, thanks to the axiomatic context being invariant under duality, allows one to establish the lemma by proving its one half only.

In our proofs we use an equational (algebraic) method of chasing subgroups, following the method found in \cite{Grandis}. With certain auxiliary lemmas, it is possible to develop a more geometric diagram chasing method, similar to the one used in e.g., \cite{Homology, Cftwm, Spider, diagram chasing, chasing subgroups} (see, in particular, Remark~2.8 in \cite{diagram chasing}). In the geometric approach, in a single diagram chase one only makes use of either the direct image maps, or the inverse image maps, and bottom and top subgroups respectively, but otherwise does not make use of the order of subgroup lattices (i.e., the subgroup lattices are viewed as merely pointed sets). However, it is not known whether the geometric approach can give rise to criteria for homomorphism induction for group-like structures that would be invariant under the duality of the corresponding noetherian form (note that the criteria in our Homomorphism Induction Theorem makes use of direct and inverse image maps simultaneously). Therefore, in this paper, we leave out entirely the geometric approach to diagram chasing.

We will use an arrow of the form
$\xymatrix{X\ar@{{ >}->}[r] &Y}$
to represent an injective morphism in a diagram, and we will use $\xymatrix{X\ar@{->>}[r] &Y}$ to represent a surjective morphism.

\begin{definition}
    A sequence of groups and morphisms, $G\xrightarrow{f} H \xrightarrow{g} I$, is said to be \textit{exact} at $H$ if $\mathsf{Im} f=\mathsf{Ker} g$.
\end{definition}
\begin{definition}
    A sequence of groups and morphisms, $X_1\rightarrow X_2\rightarrow \dots \rightarrow X_{n-1} \rightarrow X_n $, is \textit{exact} if it is exact at each node in the sequence except the end nodes $X_1$ and $X_n$.
\end{definition}

According to Mac~Lane~\cite{Homology}, the following diagram lemma was originally formulated by J.~Leicht. In~\cite{Homology} it was stated and proved for the context of modules.

\begin{theorem}[Four Lemma]
    Consider a commutative diagram
    \begin{center}
        $\xymatrix{A\ar[r]^f \ar@{->>}[d]^s &B \ar[r]^g \ar[d]^t &C \ar[r]^h \ar[d]^u &D\ar@{{ >}->}[d]^v \\
                A' \ar[r]^x & B' \ar[r]^y &C' \ar[r]^z &D',}$
    \end{center}
    where the rows are exact sequences. Then:\begin{enumerate}
        \item $g(\mathsf{Ker }t)=\mathsf{Ker }u$,
        \item $y^{-1}(\mathsf{Im }u)=\mathsf{Im }t$.
    \end{enumerate}
    In particular, if $t$ is injective, then so is $u$, and when $u$ is surjective, so is $t$ (Weak Four Lemma).
\end{theorem}
\begin{proof}
    (i) holds because
    \begin{align*}
        \Ker u & = u^{-1} 1 \wedge u^{-1}z^{-1} 1 & \text{[$u^{-1} 1\subseteq u^{-1} z^{-1} 1$]} \\
               & =u^{-1} 1\wedge h^{-1} v^{-1}1   & \text{[by commutativity]}                    \\
               & = u^{-1} 1\wedge h^{-1} 1        & \text{[$v$ is injective]}                    \\
               & = u^{-1} 1\wedge \im g           & \text{[by exactness of the top row]}         \\
               & = gg^{-1} u^{-1}1                                                               \\
               & = gt^{-1}y^{-1}1                 & \text{[by commutativity]}                    \\
               & =gt^{-1} xA'                     & \text{[by exactness of the bottom row]}      \\
               & = gt^{-1} xsA                    & \text{[$s$ is surjective]}                   \\
               & = gt^{-1} tfA                    & \text{[by commutativity]}                    \\
               & = g(fA \vee \Ker t)                                                             \\
               & = gfA \vee g\Ker t                                                              \\
               & =g\Ker t.                        & \text{[by exactness of the top row]}.
    \end{align*}
    Statement (ii) follows from (i) by duality.
\end{proof}

The following classical diagram lemma goes back to \cite{Cartan-Eilenberg,Buc,Homology}.


\begin{theorem}[Five Lemma]\label{5lem}
    Consider a commutative diagram,
    \begin{center}
        $\xymatrix{A\ar[r]^f \ar[d]^s &B \ar[r]^g \ar[d]^t &C \ar[r]^h \ar[d]^u &D\ar[d]^v \ar[r]^m &E\ar[d]^w\\
                A' \ar[r]^x & B' \ar[r]^y &C' \ar[r]^z &D'\ar[r]^n&E',}$
    \end{center}
    where the rows are exact sequences. If $t$ and $v$ are isomorphisms, $s$ is surjective and $w$ is injective, then $u$ is an isomorphism. In more detail,
    \begin{enumerate}
        \item If $s$ is surjective and $t$ and $v$ are injective, then $u$ is injective.
        \item If $w$ is injective and $t$ and $v$ are surjective, then $u$ is surjective.
    \end{enumerate}
\end{theorem}
\begin{proof}
    By Lemma \ref{classificationiso}, $t$ and $v$ are injective and surjective. Therefore by applying both parts of The Four Lemma, we obtain that $u$ is both injective and surjective, and hence an isomorphism.
\end{proof}

\begin{definition}
    A sequence of groups and morphisms, $A\xrightarrow{f}B\xrightarrow{g}C$, is said to be \textit{short exact}
    if $f$ is injective, $g$ is surjective and if $\im f=\Ker g$.
\end{definition}

The $3\times3$ Lemma appears in Buchsbaum's early approach~\cite{Buc} to what is now called an abelian category.
In \cite{Bourn2001}, Bourn gives a version in a context which encompasses semi-abelian categories. The proof given in \cite{Grandis} is identical to the one below. The formulation of the $3\times 3$ Lemma, in the context of this paper, along with a partial proof is given in \cite{workshop notes}.

\begin{theorem}[$3\times3$ Lemma]\label{33lem}
    Consider a commutative diagram
    \[
        \xymatrix{ A\ar[r]^f \ar@{{ >}->}[d]^s &B \ar[r]^g \ar@{{ >}->}[d]^t &C  \ar@{{ >}->}[d]^u  \\
            A' \ar[r]^x \ar@{->>}[d]^i& B' \ar[r]^y \ar@{->>}[d]^j&C' \ar@{->>}[d]^k  \\
            A'' \ar[r]^m  &B''\ar[r]^n  &C''}
    \]
    where each column is a short exact sequence.
    \begin{enumerate}
        \item If the bottom two rows are  short exact, then so is the top,
        \item If the top two rows are short exact, then so is the bottom.
    \end{enumerate}
\end{theorem}

\begin{proof}

    (i) We will show that  $\mathsf{Ker}f=1$ (i.e., $f$ is injective).
    \begin{align*}
        \Ker f & =f^{-1}1                                                    \\
               & =f^{-1}t^{-1}1 &  & \text{[$t$ is injective]}               \\
               & =s^{-1}x^{-1}1 &  & \text{[by commutativity]}               \\
               & =s^{-1}1       &  & \text{[the middle row is short exact]}  \\
               & =1             &  & \text{[the left column is short exact]}\end{align*}

    We will show $gB=C$ (i.e., $g$ is surjective).
    \begin{align*}
        C & =C\vee \Ker u                &  & \text{[$\Ker u \subseteq C$]}             \\
          & =u^{-1}uC                                                                   \\
          & =u^{-1}k^{-1}1               &  & \text{[the right column is short exact]}  \\
          & =u^{-1}(k^{-1}1 \wedge C')                                                  \\
          & =u^{-1}(k^{-1}1\wedge \im y) &  & \text{[the middle row is short exact]}    \\
          & =u^{-1}yy^{-1}k^{-1}1                                                       \\
          & =u^{-1}yj^{-1}n^{-1}1        &  & \text{[by commutativity]}                 \\
          & =u^{-1}yj^{-1}mA''           &  & \text{[the bottom row is exact]}          \\
          & =u^{-1}yj^{-1}miA'           &  & \text{[the left column is short exact]}   \\
          & =u^{-1}yj^{-1}jxA'           &  & \text{[by commutativity]}                 \\
          & =u^{-1}y(xA'\vee \Ker j)                                                    \\
          & =u^{-1}y(xA'\vee tB)         &  & \text{[the middle column is short exact]} \\
          & =u^{-1}(yxA' \vee ytB)                                                      \\
          & =u^{-1}(1\vee ytB)           &  & \text{[the middle row is short exact]}    \\
          & =u^{-1}ytB                                                                  \\
          & =u^{-1}ugB                   &  & \text{[by commutativity]}                 \\
          & =gB                          &  & \text{[the right column is short exact]}
    \end{align*}
    We will show $\mathsf{Im}f=\mathsf{Ker}g$ (i.e., the top row is exact).
    \begin{align*}
        \Ker g & =g^{-1}1                                                                   \\
               & =g^{-1}u^{-1}1              &  & \text{[the right column is short exact]}  \\
               & =t^{-1}y^{-1}1              &  & \text{[by commutativity]}                 \\
               & =t^{-1}xA'                  &  & \text{[the middle row is short exact]}    \\
               & =t^{-1}t(t^{-1}xA')         &  & \text{[the middle column is short exact]} \\
               & =t^{-1}(tt^{-1}xA')                                                        \\
               & =t^{-1}(\im x \wedge \im t)                                                \\
               & =t^{-1}xx^{-1}tB                                                           \\
               & =t^{-1}xx^{-1}j^{-1}1       &  & \text{[the middle column is short exact]} \\
               & =t^{-1}xi^{-1}m^{-1}1       &  & \text{[by commutativity]}                 \\
               & =t^{-1}xi^{-1}1             &  & \text{[the bottom row is short exact]}    \\
               & =t^{-1}xsA                  &  & \text{[the left column is short exact]}   \\
               & =t^{-1}tfA                  &  & \text{[by commutativity]}                 \\
               & =fA                         &  & \text{[the middle column is short exact]} \\
               & =\im f
    \end{align*}

    (ii) is dual to (i) and hence does not require a separate proof.
   
\end{proof}

One of the earliest appearances of the Snake Lemma is in \cite{Buc}. The name is derived from the shape of the construction of the connecting morphism in the resulting exact sequence. The Snake Lemma is the main tool used to construct the \textit{long exact homology sequence} (see \cite{Homology} for instance) which allows one to extract information about higher-order homology objects from lower-order homologies.
\begin{theorem}[Snake Lemma]
    Given a commutative diagram
    $$\xymatrix{A\ar[r]^f\ar[d]^{\alpha}&B\ar@{->>}[r]^g\ar[d]^{\beta}&C\ar[d]^{\gamma}\\
            A'\ar@{{ >}->}[r]^{f'}&B' \ar[r]^{g'}&C'&}$$\\
    with the rows exact, if $\mathsf{Ker}\alpha$, $\Ker \beta$, $\Ker \gamma$ are conormal and $\im \alpha$, $\im \beta$, $\im \gamma$ are normal, then there exists an exact sequence
    $$\xymatrix{\Ker \alpha \ar[r]^{\bar{f}}&\Ker \beta \ar[r]^{\bar{g}}&\Ker \gamma\ar[r]^{\delta}&\mathsf{Coker}\alpha \ar[r]^{\bar{f}'}&\mathsf{Coker}\beta \ar[r]^{\bar{g}'} &\mathsf{Coker}\gamma.}$$
\end{theorem}
\begin{proof}
    Let us first construct each morphism in the required sequence. We will do so in each case by carrying out chasing required by Theorem~\ref{HIT} along an appropriate zigzag.

    To get the morphism $\bar{f}$, consider the zigzag
    $$\Ker \alpha \xrightarrow{\iota_{\Ker \alpha}} A \xrightarrow{f} B \xleftarrow{\iota_{\Ker \beta}} \Ker \beta.$$

    \begin{align*}
        \iota_{\Ker \beta}^{-1} f\iota_{\Ker \alpha}1               & =\iota_{\Ker \beta}^{-1}1                                                                        \\&=1&&\text{[$\iota_{\Ker \beta}$ is injective]}\\
        \iota_{\Ker \alpha}^{-1}f^{-1} \iota_{\Ker \beta}\Ker \beta & =\iota_{\Ker \alpha}^{-1}f^{-1}\Ker \beta      &  & \text{[$\im \iota_{\Ker \beta}=\Ker \beta$]} \\
                                                                    & =\iota_{\Ker \alpha}^{-1} \alpha^{-1} f'^{-1}1 &  & \text{[by commutativity]}                    \\
                                                                    & =\iota_{\Ker \alpha}^{-1}\Ker \alpha           &  & \text{[$f'$ is injective]}                   \\
                                                                    & =\Ker \alpha
    \end{align*}

    To obtain the morphism $\bar{g}$, consider the zigzag
    $$\Ker \beta \xrightarrow{\iota_{\Ker \beta}} B\xrightarrow{g}C \xleftarrow{\iota_{\Ker \gamma}} \Ker \gamma.$$
    \begin{align*}
        \iota_{\Ker \gamma}^{-1} g\iota_{\Ker \beta}1                & =\iota_{\Ker \gamma}^{-1}1                                                                                                                           \\
                                                                     & =1                                                                                      &  & \text{[$\iota_{\Ker \gamma}$ is injective]}             \\
        \iota_{\Ker \beta}^{-1}g^{-1}\iota_{\Ker \gamma} \Ker \gamma & =\iota_{\Ker \beta}^{-1}g^{-1} \Ker \gamma                                              &  & \text{[$\im \iota_{\Ker \gamma}=\Ker \gamma$]}          \\
                                                                     & =\iota_{\Ker \beta}^{-1} \beta^{-1}g'^{-1} 1                                            &  & \text{[by commutativity]}                               \\
                                                                     & =\iota_{\Ker \beta}^{-1} \beta^{-1}g'^{-1} 1 \vee  \iota_{\Ker \beta}^{-1} \beta^{-1} 1 &  & \text{[$\beta ^{-1} 1\subseteq \beta ^{-1} g^{'-1} 1$]} \\
                                                                     & =\Ker \beta
    \end{align*}

    To find the morphism $\delta$, consider the zigzag
    $$\Ker \gamma \xrightarrow{\iota_{\Ker \gamma}}C \xleftarrow{g} B\xrightarrow{\beta }B' \xleftarrow{f'} A' \xrightarrow{\pi_{\im \alpha}} \mathsf{Coker} \alpha .$$
    \begin{align*}
        \pi_{\im \alpha} f'^{-1} \beta g^{-1} \iota_{\Ker \gamma} 1 & =\pi_{\im \alpha} f'^{-1} \beta g^{-1} 1                                 \\
                                                                    & =\pi_{\im \alpha} f'^{-1} \beta fA       &  & \text{[$g^{-1} 1=fA$]}     \\
                                                                    & =\pi_{\im \alpha} f'^{-1} f' \alpha A    &  & \text{[by commutativity]}  \\
                                                                    & = \pi_{\im \alpha} \alpha A              &  & \text{[$f'$ is injective]} \\
                                                                    & =1
    \end{align*}
    Dually, $$\iota_{\Ker \gamma}^{-1}g \beta^{-1} f' \pi_{\im \alpha}^{-1} \mathsf{Coker}\alpha = \iota_{\Ker \gamma}^{-1}g \beta^{-1} f' A'=\Ker \gamma.$$


    $\bar{f}'$ is constructed dually to $\bar{f}$. Further note that $\bar{f}'$ is induced by the zigzag $$\mathsf{Coker}\alpha \xleftarrow{\pi_{\im \alpha}}A'\xrightarrow{f'}B' \xrightarrow{\pi_{\im \beta}} \mathsf{Coker}\beta .$$  The morphism $\bar{g}'$ constructed dually to $\bar{g}$.

    We then have the sequence $$\xymatrix{\Ker \alpha \ar[r]^{\bar{f}}&\Ker \beta \ar[r]^{\bar{g}}&\Ker \gamma\ar[r]^{\delta}&\mathsf{Coker}\alpha \ar[r]^{\bar{f}'}&\mathsf{Coker}\beta \ar[r]^{\bar{g}'} &\mathsf{Coker}\gamma .}$$Let us now show that this sequence is exact. Note that by the Homomorphism Induction Theorem, to compute the kernels and images of these morphisms, it suffices to chase the smallest subgroup backward and the largest subgroup forward, respectively, along the zigzags that were used to construct those morphisms.
    \begin{align*}
        \im \bar{f} & =\iota_{\Ker \beta }^{-1} f\iota_{\Ker \alpha} \Ker \alpha                                                                   \\
                    & =\iota_{\Ker \beta }^{-1} f\alpha^{-1} 1                                                                                     \\
                    & =\iota_{\Ker \beta }^{-1} f\alpha^{-1}f^{'-1}1                                    &   & \text{[$f'$ is injective]}           \\
                    & =\iota_{\Ker \beta }^{-1} f f^{-1} \beta^{-1} 1                                   &   & \text{[by commutativity]}            \\
                    & =\iota_{\Ker \beta }^{-1}( \beta^{-1} 1\wedge \im f)                                                                         \\
                    & = \iota_{\Ker \beta }^{-1}( \beta^{-1} 1) \wedge (\iota_{\Ker \beta }^{-1} \im f)                                            \\
                    & = \Ker \beta \wedge (\iota_{\Ker \beta }^{-1} \Ker g)                             &   & \text{[by exactness of the top row]} \\
                    & = \iota_{\Ker \beta }^{-1} g^{-1} \iota_{\Ker \gamma} 1                           &                                          \\
                    & = \Ker \bar{g}
    \end{align*}
    \begin{align*}
        \Ker \delta & =\iota_{\Ker \gamma}^{-1}g\beta^{-1}f' \pi_{\im \alpha}^{-1}1                                    \\
                    & =\iota_{\Ker \gamma}^{-1}g\beta^{-1}f'\im \alpha                                                 \\
                    & =\iota_{\Ker \gamma}^{-1}g\beta^{-1}\beta fA                  &  & \text{[by commutativity]}     \\
                    & =\iota_{\Ker \gamma}^{-1}g(fA\vee \beta^{-1}1)                                                   \\
                    & =\iota_{\Ker \gamma}^{-1}(gfA\vee g\beta^{-1}1)                                                  \\
                    & =\iota_{\Ker \gamma}^{-1}(1\vee g\beta^{-1}1)                 &  & \text{[the top row is exact]} \\
                    & =\iota_{\Ker \gamma}^{-1}g\beta^{-1}1                                                            \\
                    & =\im \bar{g}
    \end{align*}

    It follows that $\im \delta =\Ker \bar{f}$ by applying duality to the induced morphism (this is possible thanks to Lemma~\ref{induced duality}).

    Lastly, $\im \bar{f}'=\Ker \bar{g}'$ follows by duality applied to $\im \bar{f}=\Ker \bar{g}$.
\end{proof}

\begin{remark}
    \label{snake remark}
    Note that the Snake Lemma proved here is more general than the one found in \cite{workshop notes}, which relies on the additional assumptions that the kernel of the diagonal of the right square is conormal and that the image of the diagonal of the left square is normal. While it is true, for groups, that the kernel of any morphism is conormal (since every subgroup of a given group is conormal), it is not always true that the image of a morphism is normal. We then give an example in the context of groups in which the image of the diagonal of the left square is not normal.  Consider a commutative diagram of groups
    $$\xymatrix{B \ar@{{ >}->}[r]^f\ar@{{ >}->}[d]^f&V \ar@{{ >}->}[d]^i \ar@{->>}[r]^g &  V/B\ar[d]^h \\
            V \ar@{{ >}->}[r]^i& D_8 \ar@{->>}[r]^j& D_8 /V .}$$
    The group $D_8$ is the dihedral group of order $8$ given by $D_8=\langle a,b\mid o(b)=2,o(a)=4, ba=a^{-1}b\rangle=\{e,a,a^2,a^3,b,ab,a^2 b,a^3 b\}$ (where $o(x)$ is the order of the element $x$). $B$ is the subgroup $B=\{e,b\}$ of $D_8$, while $V$ is the subgroup $V=\{e,b,a^2,a^2 b\}$ of $D_8$. Note that $B\approx  \mathds{Z}_2$ (since $B$ is of order 2), $V\approx  \mathds{V}_4$ (the Klein-four group), and $B\subseteq V$.
    The morphism $f$ is the embedding of $B$ into $V$. The morphism $i$ is similarly defined to be the embedding of the subgroup $V$ into the group $D_8$. Notice that since $B$ and $V$ are subgroups of index $2$ in $V$ and $D_8$ respectively, both $B$ and $V$ are normal in $V$ and $D_8$ respectively. Therefore, we may define the morphisms $g$ and $j$ as projections associated to these normal subgroups. Observe that the rows are then exact. Lastly, the morphism $h$ is a zero morphism. Observe that $h$ makes the right square commute since $ji$ is a zero morphism (since $\im i=\Ker j$).
    Furthermore, since the rows are exact, we may apply the Snake Lemma (from this paper) to obtain an exact sequence
    $$1\rightarrow 1 \rightarrow V/B \rightarrow V/B \rightarrow  D_8/V \rightarrow D_8 /V.$$
    Notice that $V/B$ and $D_8 /V$ are both groups of order $2$ and hence must be isomorphic to $\mathds{Z}_2$. We then obtain the exact sequence $$1\rightarrow 1\rightarrow \mathds{Z}_2 \rightarrow \mathds{Z}_2 \rightarrow \mathds{Z}_2 \rightarrow \mathds{Z}_2$$ by including the relevant isomorphisms. However, since the normalizer of the subgroup $B$ of $D_8$ is precisely~$V$ (this may be easily verified by calculation), we have that $B$ is normal in $V$ but not in $D_8$. Therefore, the image $\im (if)=B$ of the left diagonal is not normal. This shows us that the Snake Lemma in \cite{workshop notes} may not be applied to this example, while the Snake Lemma in this paper may.
\end{remark}

\section{The Middle $3\times 3$ Lemma}
\begin{theorem}
    Consider a commutative diagram
    \[
        \xymatrix{ A\ar[r]^f \ar@{{ >}->}[d]^s &B \ar[r]^g \ar@{{ >}->}[d]^t &C  \ar@{{ >}->}[d]^u  \\
            A' \ar[r]^x \ar@{->>}[d]^i& B' \ar[r]^y \ar@{->>}[d]^j&C' \ar@{->>}[d]^k  \\
            A'' \ar[r]^m  &B''\ar[r]^n  &C''}
    \]
    where each column is a short exact sequence. If the top and bottom rows are short exact and $yx$ is a zero morphism, then the middle row is short exact.
    
\end{theorem}
\begin{proof}
    $\im x \subseteq \Ker y$ since $yx$ is a zero morphism. By Axiom~\ref{ax6}, there exists a binormal subgroup $T$ of $B'$ such that $\im x \subseteq T \subseteq \Ker y$. We will show that $T \subseteq \im x$.
    \begin{align*}
        \im x &= \im x \vee \im xs\\
        &= \im x \vee \im tf\\
        &= \im x \vee t \im f\\
        &= \im x \vee t (\Ker g)\\
        &= \im x \vee t(g^{-1}u^{-1} 1)\\
        &= \im x \vee t(t^{-1}y^{-1} 1)\\
        &= \im x \vee (\im t \wedge \Ker y) \\
        & \supseteq  \im x \vee (\im t \wedge T)\\
        &= \im x \vee (\Ker j \wedge T)\\
        &= (\im x \vee \Ker j) \wedge T\\
        &= (j^{-1}j xA' ) \wedge T\\
        &= (j^{-1} miA') \wedge T\\
        &= (j^{-1} mA'') \wedge T\\
        &= (j^{1} n^{-1} 1) \wedge T \\
        &= (y^{-1} k^{-1} 1) \wedge T \\
        &\supseteq \Ker y \wedge T\\
        &=T.
    \end{align*}
It follows that $\im x =T$. Moreover, by a dual argument, it follows that $T=\Ker y$. 
\end{proof}

In fact, the Middle $3\times 3$ Lemma above is actually equivalent to the property that the subgroup poset of a group cannot have the following subposet

\[
\xymatrix@R=1.5cm@C=1.5cm{
  & *+[o][F-]{J} \ar@{-}[dl] \ar@{-}[dr] \\
  *+[o][F-]{N} \ar@{-}[d] & & *+[o][F-]{B} \ar@{-}[ddl] \\
  *+[o][F-]{C} \ar@{-}[dr] &   & \\
  &*+[o][F-]{M}&
}
\]
where $C$ is conormal, $N$ is normal and $B$ is binormal. If we do have such a subposet, then we may construct the commutative diagram 
$$\xymatrix{M \ar[r]^f \ar@{{ >}->}[d]^s &B \ar[r]^g \ar@{{ >}->}[d]^t &B/M  \ar@{{ >}->}[d]^u  \\
            C \ar[r]^x \ar@{->>}[d]^i& J \ar[r]^y \ar@{->>}[d]^j&J/N \ar@{->>}[d]^k  \\
            C/M \ar[r]^m  &J/B \ar[r]^n  &0_J.}$$
The diagram is formed by noting that: \begin{enumerate}
    \item $M\lhd B$ since it is the kernel of the composite $B \rightarrow J \rightarrow J/N$
    \item Similarly $M \lhd C$ as it the kernel of the composite $C \rightarrow J \rightarrow J/B$.
    \item The map $u$ is induced by the zigzag $$B/M  \leftarrow B \rightarrow J \rightarrow J/N.$$ It may be verified that $u$ is an isomorphism by the Universal Isomorphism Theorem. Similarly, $m$ is an isomorphism.
    \item $J= N \vee B \lhd J$ and thus we use the notation $0_J=J/J$.
\end{enumerate} 
It follows that the columns and the bottom and top rows are short exact. Thus, the Middle $3\times 3$ Lemma forces $C=N$. 
\section{Appendix: Some Exercises}

In this section we offer the reader a playground of exercises for trying out the technique of proving diagram lemmas illustrated in the previous section. While some of the exercises are new lemmas discovered by the first author in \cite{Day}, most of them are well-known to hold for various categories of group-like structures.

\begin{exercise}[Incomplete Snail Lemma]
    Given a commutative diagram,
    $$
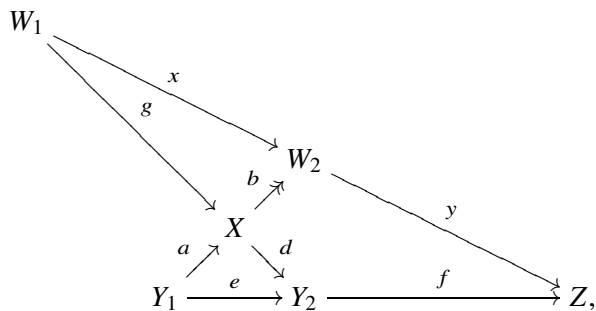
\xymatrix@!=2pt{W_1 \ar[ddrrrr]^x \ar[dddrrr]^g&&&&&&&&\\
        &&&&&&&&\\
        &&&&W_2\ar[ddrrrr]^y&&&&\\
        &&&X\ar@{->>}[ur]^b \ar[dr]^d&&&&&\\
        &&Y_1 \ar[rr]^e \ar[ur]^a &&Y_2\ar[rrrr]^f &&&&Z ,}$$
    where all sequences except possibly the sequence $W_1\rightarrow W_2 \rightarrow Z$ are exact. Prove that also the sequence ${W_1\rightarrow W_2 \rightarrow Z}$ is exact.
\end{exercise}

\begin{exercise}\label{square exact}
    Consider a commutative diagram:
    $$
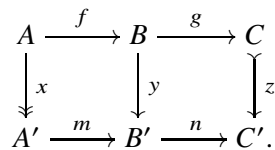
\xymatrix{A\ar[r]^f\ar@{->>}[d]^x&B\ar[r]^g\ar[d]^y&C\ar@{{ >}->}[d]^z\\
            A'\ar[r]^m&B'\ar[r]^n &C'.}$$
    Prove one of the following two statements and deduce the other by duality:
    \begin{enumerate}
        \item If $y$ is surjective and the top row is exact, then the bottom row is exact.
        \item If $y$ is injective and the bottom row is exact, then the top row is exact.
    \end{enumerate}
\end{exercise}

\begin{exercise}[Spider Lemma]
    Consider a commutative diagram,
    $$ 
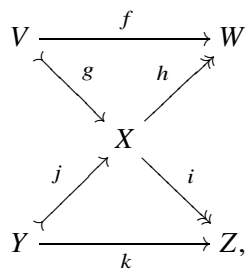
\xymatrix{V\ar[rr]^f\ar@{{ >}->}[dr]^g&&W\\
            &X\ar@{->>}[dr]^i \ar@{->>}[ur]^h&\\
            Y\ar@{{ >}->}[ur]^j \ar[rr]_k&&Z,}$$
    where the diagonals are short exact sequences. Prove that if $k$ is an isomorphism, then so is $f$. Hint: use duality to shorten the proof.
\end{exercise}

\begin{exercise}\label{diamond lemma}
    Consider a commutative diagram:
    $$
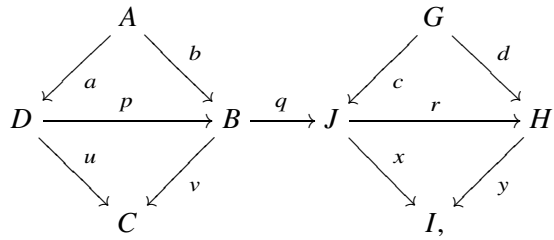
\xymatrix{&A\ar[dl]^a\ar[dr]^b&&&G\ar[dl]^c\ar[dr]^d&\\ D \ar[rr]^p\ar[dr]^u&&B\ar[r]^q\ar[dl]^v&J\ar[rr]^r\ar[dr]^x&&H\ar[dl]^y\\
            &C&&&I,&
        }$$
    where the horizontal sequence is exact. Prove the following:
    \begin{enumerate}
        \item If $u$ is surjective and $v$ and $y$ are injective, then $x$ is injective.
        \item If $a$ and $c$ are surjective and $d$ is injective, then $b$ is surjective.
    \end{enumerate}
\end{exercise}


\begin{exercise}[Diamond Lemma]\label{diamond 1}
    Consider a commutative diagram,
    $$
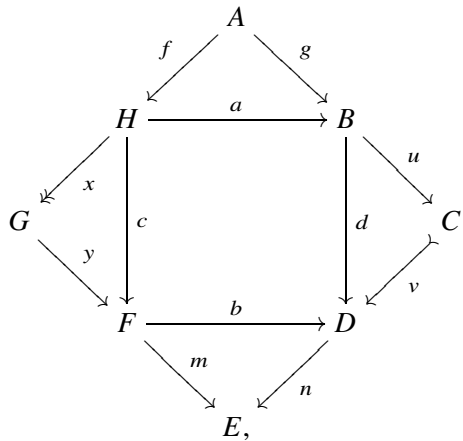
\xymatrix{&&&A\ar[dl]_f\ar[dr]^g&&&\\
            &&H\ar[rr]^a\ar[dd]^c\ar@{->>}[dl]^x &&B\ar[dd]^d\ar[dr]^u&&\\
            &G\ar[dr]^y&&&&C\ar@{{ >}->}[dl]^v&\\
            &&F\ar[rr]^b\ar[dr]^m&&D\ar[dl]^n &&\\
            &&&E,&&&}$$
    where the four edges of the outer square are exact sequences. (The sequence $(y,m)$, as well as the other sequences forming edges of the outer square are assumed to be exact and not necessarily short exact.) Moreover, the arrows $x$ and $v$  are surjective and injective, respectively.
    Prove the following:
    \begin{enumerate}
        \item If $a$ is injective, then $y$ is injective.
        \item If $b$ is surjective, then $u$ is surjective.
    \end{enumerate}
\end{exercise}

\begin{exercise}[Baby Dragon Lemma]\label{baby dragon}
    Consider a commutative diagram,
    $$
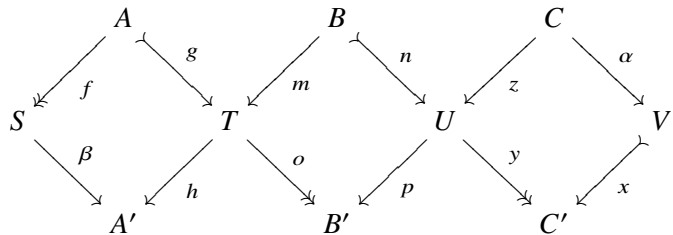
\xymatrix{
            &&A \ar@{->>}[dl]^f \ar@{{ >}->}[dr]^g &&B\ar@{{ >}->}[dr]^n\ar[dl]^m &&C\ar[dl]^z\ar[dr]^{\alpha}&&\\
            &S\ar[dr]^{\beta}&&T\ar[dl]^h \ar@{->>}[dr]^o&&U\ar[dl]^p\ar@{->>}[dr]^y&&V\ar@{{ >}->}[dl]^x&\\
            &&A'&&B'&&C'&&\\
        }$$
    where all the diagonal sequences are exact, and certain arrows are injective or surjective as indicated.  Prove the following statements:
    \begin{enumerate}
        \item If $\alpha$ is injective, then so is $\beta$.
        \item If $\beta$ is surjective, then so is $\alpha$.
    \end{enumerate}
\end{exercise}

\begin{remark}
    Note that this diagram lemma holds for arbitrarily long diagrams of this type (i.e., we may insert as many copies of the middle diamond as we wish), as in the case of the Dragon Lemma below.
\end{remark}


\begin{exercise}[Dragon Lemma]
    Consider a commutative diagram
    $$
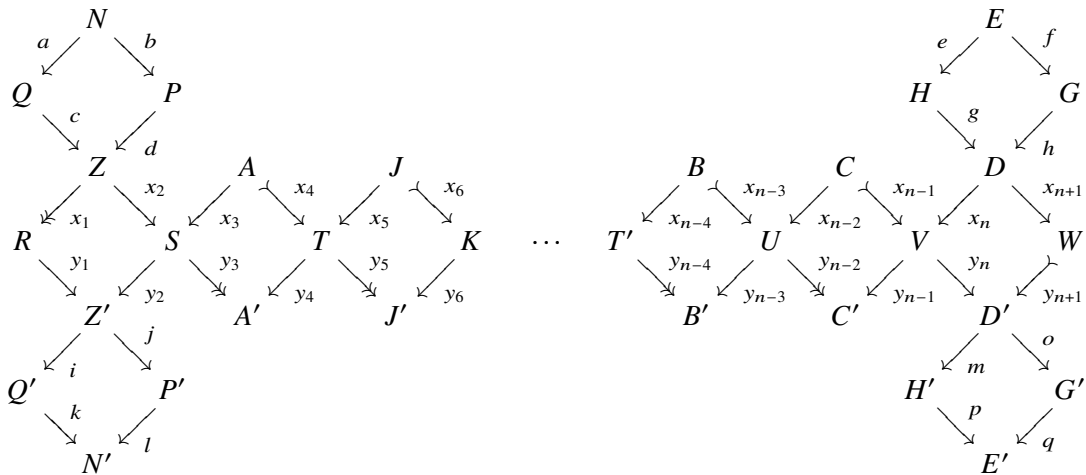
\xymatrix@!=4pt{&&N\ar[dl]_{a} \ar[dr]^{b}&&&&&&&&&&&&E\ar[dr]^{f}\ar[dl]_{e}&&\\
        &Q\ar[dr]^{c}&&P\ar[dl]^{d}&&&&&&&&&&H\ar[dr]^{g}&&G\ar[dl]^{h}&\\
        &&Z\ar[dr]^{x_2} \ar@{->>}[dl]^{x_1} &&A \ar[dl]^{x_3} \ar@{{ >}->}[dr]^{x_4} &&J\ar[dl]^{x_5} \ar@{{ >}->}[dr]^{x_6} &&&&B\ar@{{ >}->}[dr]^{x_{n-3}}\ar[dl]^{x_{n-4}} &&C\ar[dl]^{x_{n-2}}\ar@{{ >}->}[dr]^{x_{n-1}}&&D\ar[dl]^{x_{n}} \ar[dr]^{x_{n+1}} &&\\
        &R \ar[dr]^{y_1}&&S\ar@{->>}[dr]^{y_3}\ar[dl]^{y_2}&&T\ar[dl]^{y_4} \ar@{->>}[dr]^{y_5}&&K\ar[dl]^{y_6}&\cdots&T'\ar@{->>}[dr]^{y_{n-4}}&&U\ar[dl]^{y_{n-3}}\ar@{->>}[dr]^{y_{n-2}}&&V\ar[dl]^{y_{n-1}}\ar[dr]^{y_{n}}&&W\ar@{{ >}->}[dl]^{y_{n+1}}&\\
        &&Z'\ar[dl]^{i} \ar[dr]^{j}&&A'&&J'&&&&B'&&C'&&D'\ar[dr]^{o}\ar[dl]^{m}&&\\
        &Q'\ar[dr]^{k}&&P'\ar[dl]^{l}&&&&&&&&&&H'\ar[dr]^{p}&&G'\ar[dl]^{q}&\\
        &&N'&&&&&&&&&&&&E'&&}$$
    where all sequences are exact, and certain arrows are injective or surjective as indicated. Prove the following:
    \begin{enumerate}
        \item If $a$ and $e$ are surjective, then $y_1$ is injective.
        \item If $l$ and $q$ are injective, then $x_{n+1}$ is surjective.
    \end{enumerate}
\end{exercise}

A version of the following diagram lemma was formulated and proved by J.~Lambek in \cite{Lambek} for the context of groups.

\begin{exercise}[Goursat's Theorem]\label{goursatex}
    Given a commutative diagram,
    $$
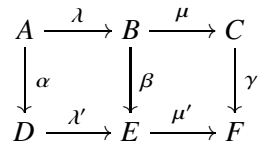
\xymatrix{A\ar[r]^{\lambda}\ar[d]^{\alpha}&B\ar[r]^{\mu}\ar[d]^{\beta}&C\ar[d]^{\gamma}\\
            D \ar[r]^{\lambda '}  &E \ar[r]^{\mu '}& F}$$\\
    with exact rows, prove that if $\Ker (\gamma \mu)$ is conormal, then:
    \begin{enumerate}
        \item $\im (\beta \lambda)  \lhd \im \beta \wedge \im \lambda '$ and $\Ker \beta \vee \Ker \mu \lhd   \Ker (\gamma \mu)$, and
        \item  $(\im \beta \wedge \im \lambda ')/ \im (\beta \lambda) \approx \Ker (\gamma \mu) /(\Ker \beta \vee \Ker \mu )$.
    \end{enumerate}
    Hint: use Theorem \ref{quotientiso}.
\end{exercise}

\begin{exercise}[Generalized Snail Lemma]\label{generalizedsnail}
    Given a commutative diagram
    $$\xymatrix{A\ar[dd]^{\alpha} \ar[rr]^f  \ar[dr]^{\gamma}&&B\ar[dd]^{\beta}\\
            &C\ar[ur]^{f_0 '}\ar[dl]^{\beta '}&\\
            A_0\ar[rr]^{f_0}&&B_0 ,}$$\\
    if $\Ker \gamma$, $\Ker \alpha$, $\Ker \beta '$ are conormal and $\im \gamma$, $\im \alpha$, $\im \beta '$ are normal, then prove that the following sequence is exact:
    $$\Ker \gamma \xrightarrow{v} \Ker \alpha
        \xrightarrow{w} \Ker \beta '
        \xrightarrow{x} \mathsf{Coker} \gamma
        \xrightarrow{y} \mathsf{Coker} \alpha
        \xrightarrow{z}\mathsf{Coker} \beta'.$$
\end{exercise}

\begin{definition}
    A \textit{double complex} is a triple $(X,\delta_h ,\delta_v)$, where for all integers $m$ and $n$, $X=(X^{n,m})$ is a family of groups,
    $\delta_h=(\delta_h^{n,m}\colon X^{n,m}\rightarrow X^{n,m+1})$ and $\delta_v=(\delta_v^{n,m}\colon X^{n,m}\rightarrow X^{n+1,m})$ are families of group morphisms such that $\delta_h^{n,m} \delta_h^{n,m-1}=0$, (i.e., $\delta_h^{n,m} \delta_h^{n,m-1}$ is a zero morphism), $\delta_v^{n,m}\delta_v^{n-1,m}=0$ (i.e., $\delta_v^{n,m}\delta_v^{n-1,m}$ is a zero morphism), and $\delta_v^{n,m+1} \delta_h^{n,m}=\delta_h^{n+1, m} \delta_v^{n,m}$.
\end{definition}
\begin{figure}
    \centering
    $$\xymatrix{&&\ar@{.}[d]&\ar@{.}[d]& &\\
            &&\bullet\ar[d]^m &\ar[d] &&\\
            \ar@{.}[r]&\bullet \ar[r]^a \ar[d]^j \ar[dr]^p &C \ar[d]^c \ar[dr]^r \ar[r]^k & \bullet \ar[d]^v \ar[r] &\ar@{.}[r]&\\
            \ar@{.}[r]&\bullet \ar[r]^d &A \ar[r]^e \ar[d]^f \ar[dr]^q &B\ar[d]^g \ar[r]^s& \bullet \ar[d]^n \ar@{.}[r]&\\
            \ar@{.}[r]&\ar[r]&\bullet \ar[d]\ar[r]^l &D\ar[d]^u \ar[r]^t&\bullet \ar@{.}[r]&,\\
            &&\ar@{.}[d]&\bullet \ar@{.}[d]&&\\
            &&&&& } $$
    \caption{Double Complex}
    \label{fig:double complex}
\end{figure}
\begin{definition}
    Consider a double complex in Figure~\ref{fig:double complex}
    (where $p=ca$, $r=ec$ and $q=ge$). We define the following \textit{homology objects} associated with the group $A$:
    \begin{enumerate}
        \item $ A_h=\Ker e/ \im d$, whenever $\im d \lhd \Ker e$.
        \item $ A_{\square}=\Ker q /(\im c \vee \im d)$, whenever $\im c \vee \im d \lhd \Ker q$.
        \item $ {}^{\square}A=(\Ker e\wedge \Ker f)/\im p$, whenever $\im p \lhd \Ker e \wedge \Ker f$.
    \end{enumerate}
\end{definition}
\begin{theorem}[Salamander Lemma]
    Consider a double complex in Figure~\ref{fig:double complex}.
    If the homology objects $C_{\square}$, $A_h$, $A_{\square}$, ${}^{\square}\!B$, $B_h$ and ${}^{\square}\!D$ are defined, and $\im c$ is a normal subgroup of $A$, then there is an exact sequence:
    $$C_{\square}\xrightarrow{v} A_h \xrightarrow{w} A_{\square}\xrightarrow{x} {}^{\square}\!B\xrightarrow{y} B_h \xrightarrow{z} {}^{\square}\!D.$$
\end{theorem}


The Incomplete Snail Lemma (ISL, for short), first stated in the context of a pointed regular category in \cite{Snail}, is a stepping stone towards the Snail Lemma due to E.~Vitale \cite{Snail 1}. As shown in \cite{Snail}, in a pointed regular category, the (Complete) Snail Lemma holds if and only if ISL holds and every kernel has a cokernel. The main result obtained in \cite{Snail} was that in this context, ISL is equivalent to subtractivity \cite{subtractivity}. Combining this with the results of \cite{pointed subobject functor, Spider}, ISL is further equivalent to the Lower $3\times 3$ (in our context, that would be Part~(ii) of our Lemma~\ref{33lem}), while in the context of a normal category (a regular category where every regular epimorphism is a cokernel), it is equivalent to the Snake Lemma, as well as the Upper $3\times 3$ Lemma (Part~(i) of our Lemma~\ref{33lem}). With a suitable formulation of the Snake Lemma, it is furthermore shown in \cite{Spider} that the Snake Lemma holds in a pointed regular category if and only if the category is normal and subtractive. For comparison: in the same context, the Five Lemma is equivalent to the Middle $3\times 3$ Lemma, and is in turn equivalent to \emph{protomodularity} in the sense of D.~Bourn \cite{Bourn}, which is known to imply both normality and subtractivity. The Snail Lemma was established in \cite{Snail} in the very context of a pointed regular category that is protomodular (such category is also called a \emph{homological category} in \cite{Borceux Bourn}). In \cite{SnailApp}, the Snail Lemma was used to generalize two exact sequences arising from a pointed functor between pointed groupoids and a fibration of pointed groupoids respectively. This was done by replacing groupoids by groupoids internal to a suitable category. This was generalized even further in \cite{SnailApp 1} by considering profunctors between internal groupoids. Note that the exercise on the Generalized Snail Lemma, which yields a six term exact sequence similar to that of the Snake Lemma, becomes equivalent to the Snail Lemma in the presence of finite limits, which we do not require in our context. It is also worth noting that the `form of subobjects' of a homological category is not, in general, noetherian. Indeed, this would force the category to be semi-abelian (see \cite{DNA 0}), while not every homological category is semi-abelian (a well-known example is the category of topological groups --- see \cite{Borceux Bourn}, for instance). Thus, the exercises in this section do not directly recover the Snail Lemma and ISL in the original settings where they have been first established (unlike in the case of the classical diagram lemmas). Similarly, the exercises on the Spider Lemma and the Dragon Lemma  do not recover these results in the context of a subtractive regular category, which is where they were first established.

The Spider and the Dragon Lemmas introduced in \cite{Spider, diagram chasing} are in a sense the shortest and the longest non-trivial diagram lemmas that can be proved in a pointed regular subtractive category \cite{subtractivity, pointed subobject functor}, and hence in any semi-abelian category. In fact, as shown in \cite{Spider}, for a pointed regular category, the Spider Lemma is equivalent to subtractivity. It is likely that a similar fact is true for the Dragon Lemma.


Goursat's Theorem for the context of groups provides a one to one correspondence between subgroups of the product of two groups and isomorphisms between subquotients of the two groups. The homological lemma given in Exercise~\ref{goursatex} was formulated in \cite{Lambek} and was shown to be equivalent to Goursat's Theorem in the context of groups. It was shown in \cite{Lambek, Lam1} that the connecting morphism in the Snake Lemma, the Zassenhaus-Schreier Refinement Theorem and the Jordan-Hölder Theorem, for the context of groups, follow from different formulations of Goursat's Theorem.


The Salamander Lemma, which was originally formulated in \cite{Bergman} in the context of abelian categories, may also be generalized to our context. The Salamander Lemma was formulated and proved in our context in \cite{Salamander}, where the proof was simplified using two results (see Proposition 3.1 and Proposition 3.5 in \cite{Salamander}) to show the existence of the morphisms in the sequence and a simplified condition to show exactness of these morphisms. The Salamander Lemma produces an exact sequence of homology objects from a double complex.  A number of classical diagram lemmas, including the Four Lemma, Five Lemma, $3\times 3$ Lemma and the Snake Lemma, all follow as direct consequences of the Salamander Lemma (see \cite{Bergman}) for the context of abelian categories. However, in our context, the derivation of these classical diagram lemmas from the Salamander Lemma does not work, since it is not guaranteed that the needed homology objects will exist.

\begin{remark} Note that there is an error in the proof of Theorem~4.2 in \cite{Salamander}, claiming that the $3\times 3$ Lemma, without any additional assumptions of existence of homology objects, follows from the Salamander Lemma.
\end{remark}

\begin{exercise}[Short Five Lemma]
    Consider a commutative diagram,
    \begin{center}
        $\xymatrix{ A\ar@{{ >}->}[r]^f \ar[d]^s  &B \ar@{->>}[r]^g \ar[d]^t  &C\ar[d]^u \\
                A' \ar@{{ >}->}[r]^x & B' \ar@{->>}[r]^y &C',}$
    \end{center}
    where the rows are short exact sequences. Prove the following statements:
    \begin{enumerate}
        \item If $s$ and $u$ are injective then so is $t$.
        \item If $s$ and $u$ are surjective, then so is $t$.

        \item If $s$ and $u$ are isomophisms, then so is $t$.
    \end{enumerate}
\end{exercise}

\begin{remarks}
    (1) The special case where $s$ and $u$ are identities expresses that whenever we have two short exact sequences `from $A$ to $C$', either they are connected by a map $t$ as in the diagram --- in which case they are isomorphic --- or no map $t$ making the diagram commute exists. This is a historically important ingredient towards the interpretation of cohomology via isomorphism classes of short exact sequences; see for instance~\cite{Cartan-Eilenberg} where its use is implicit, or~\cite{Homology} which views the lemma as a result on its own.

    (2) Alternatively, specifying to the situation where $u$ is an identity while the surjections $g$ and $y$ are split epimorphisms in the context of a pointed category yields a version of Bourn's protomodularity~\cite{Bourn}, one of the key ingredients in the definition of a semi-abelian category.
\end{remarks}


Note that short exact sequences (as defined in the previous section) are precisely the exact sequences of the form $$\xymatrix{A\ar@{{ >}->}[r]^{f} &B\ar@{->>}[r]^-{g} &C.}$$
However, in all known contexts of interest, a short exact sequence is equivalent to an exact sequence consisting of five groups where the two outer groups have exactly one subgroup (this is taken as the definition of a short exact sequence for modules in \cite{Homology}); in other words, an exact sequence
$$\xymatrix{0\ar[r] & A \ar[r]^{f} &B\ar[r]^-{g} & C\ar[r] & 0',}$$
where $0$, $0'$ are groups each having exactly one subgroup --- we will refer to these as \emph{strongly short exact sequences}. We will call a group \emph{trivial} when it has exactly one subgroup.

\begin{exercise}
    Show that in a strongly short exact sequence, the groups $0$ and $0'$ are necessarily isomorphic (more generally, show that any two trivial groups connected by a zigzag of morphisms are isomorphic). Moreover, show that in every strongly short exact sequence, the pair $f,g$ is a short exact sequence.
\end{exercise}

The equivalence of the two approaches of defining short exact sequences follows from the following additional axiom.

\begin{axiom}\label{ax6} For every group, its smallest subgroup is conormal while its largest subgroup is normal.
\end{axiom}

\begin{exercise}
    Show that Axiom~\ref{ax6} is equivalent to requiring that a morphism $f\colon A\rightarrow B$ is injective if and only if it forms an exact sequence $$\xymatrix{0\ar[r] &A\ar[r]^{f}& B}$$
    for some trivial $0$, and, along with this, requiring that dually, $f$ is surjective if and only if it forms an exact sequence $$\xymatrix{A\ar[r]^{f}&B \ar[r] &0}$$
    for some trivial $0$.
    Deduce from this that under Axiom~\ref{ax6}, every short exact sequence completes to a strongly short exact sequence.
\end{exercise}

\begin{exercise}
    Formulate and prove the diagram lemmas by replacing short exact sequences with strongly short exact sequences, and by replacing injective/surjective morphisms with those that complete to sequences as shown in the previous exercise (but do this, without assuming Axiom~\ref{ax6}). Show that the version of the Short Five Lemma where the horizontal sequences have been replaced with strongly short exact sequences is a consequence of the Five Lemma. Establish a similar link between the Baby Dragon Lemma and the Dragon Lemma.
\end{exercise}

\begin{remark}
    A context where all axioms hold except Axiom~\ref{ax6} is that of rings with identity and identity-preserving ring homomorphisms, where subgroups are additive subgroups of rings (see \cite{DNA IV}). Conormal subgroups there are subrings containing the identity element, while normal subgroups are ideals. The largest additive subgroup of a ring is, of course, an ideal, but the dual property fails: the smallest additive subgroup does not contain the identity element. While this example shows that Axiom~\ref{ax6} is independent of the rest of the axioms, it is not a particularly interesting example from the point of view of homological diagram lemmas, since only the largest subgroups are both normal and conormal (an ideal contains the identity element only when it is the entire ring), and hence there is a shortage of exact sequences. We do not know of a naturally arising example where Axiom~\ref{ax6} fails, but at the same time, where there isn't a shortage of short exact sequences.
\end{remark}





\section*{Acknowledgements} The first author wishes to thank Stellenbosch University for its kind hospitality during his visits in October and November 2022. The fourth author wishes to thank the African Institute for Mathematical Sciences (AIMS) and Stellenbosch University for their kind hospitality and the opportunity to present an invited lecture series in the \emph{1st Workshop on Mathematical Structures} for which the notes \cite{workshop notes} were prepared.
\\\\
\textbf{Funding Declaration}
The authors have not received any funding for this project.

\end{document}